\documentclass[a4paper]{amsart}
\pdfoutput=1

\usepackage[utf8]{inputenc}
\usepackage[T1]{fontenc}
\usepackage{lmodern}
\usepackage[english]{babel}
\usepackage{amsmath,amssymb,amsthm}
\usepackage[marginratio={1:1,1:1},totalwidth=404pt,totalheight=606pt]{geometry}
\usepackage{microtype}
\usepackage{hyperref}
\usepackage{url}
\usepackage{color}

\theoremstyle{plain}
\newtheorem{theorem}{Theorem}
\newtheorem{lemma}[theorem]{Lemma}
\newtheorem{proposition}[theorem]{Proposition}
\newtheorem{corollary}[theorem]{Corollary}
\newtheorem*{proposition*}{Proposition}
\newtheorem*{lemma*}{Lemma}

\theoremstyle{definition}
\newtheorem{definition}[theorem]{Definition}
\newtheorem{example}[theorem]{Example}
\newtheorem{remark}[theorem]{Remark}

\theoremstyle{remark}

\newcommand{\C}{\mathbb{C}}
\newcommand{\N}{\mathbb{N}}
\newcommand{\Z}{\mathbb{Z}}

\makeatletter
\renewcommand{\pod}[1]{\allowbreak\mathchoice
  {\if@display \mkern 18mu\else \mkern 8mu\fi (#1)}
  {\if@display \mkern 18mu\else \mkern 8mu\fi (#1)}
  {\mkern4mu(#1)}
  {\mkern4mu(#1)}
}

\numberwithin{theorem}{section}

\title[$C^*$-simplicity, amalgams, and radical classes]{$C^*$-simplicity of free products with amalgamation and radical classes of groups}

\author{Nikolay A.\ Ivanov}
\address{Faculty of Mathematics and Informatics\\University of Sofia\\blvd.\ James Bourchier 5\\BG-1164 Sofia\\Bulgaria}
\email{nikolay.antonov.ivanov@gmail.com}

\author{Tron Omland}
\address{Department of Mathematics\\University of Oslo\\P.O.Box 1053 Blindern\\NO-0316 Oslo\\Norway}
\email{trono@math.uio.no}

\date{November 30, 2017}
\subjclass[2010]{22D25, 20E06 (Primary) 46L05, 43A07 (Secondary)}
\keywords{$C^*$-simplicity, free product of groups with amalgamation, radical classes of groups}

\begin{document}

\begin{abstract}
We give new characterizations to ensure that a free product of groups with amalgamation has a simple reduced group $C^*$-algebra,
and provide a concrete example of an amalgam with trivial kernel,
such that its reduced group $C^*$-algebra has a unique tracial state, but is not simple.

Moreover, we show that there is a radical class of groups
for which the reduced group $C^*$-algebra of any group is simple
precisely when the group has a trivial radical corresponding to this class.
\end{abstract}

\maketitle

\section{Introduction}

Groups have been among the most studied objects in connection with operator algebras,
and one of the natural questions to consider in this regard is whether group $C^*$-algebras are simple.
While full group $C^*$-algebras can never be simple (unless the group is trivial),
the problem for reduced group $C^*$-algebras has been of great interest ever since the work of Powers \cite{Powers},
showing that nonabelian free groups have a simple reduced group $C^*$-algebra.
Larger classes of groups with properties similar to the one Powers described were later studied,
and several results and open problems on this topic are discussed by de~la~Harpe \cite{Harpe2}.

Recall that a discrete group $G$ is called $C^*$-simple if its reduced group $C^*$-algebra $C^*_r(G)$ is simple,
and $G$ is said to have the unique trace property if $C^*_r(G)$ has a unique tracial state.
A long-standing open problem was whether the two properties were equivalent,
but fairly recently it has been shown that $C^*$-simplicity is strictly stronger than the unique trace property.
The ``stronger'' part is due to Breuillard, Kalantar, Kennedy, and Ozawa \cite{BKKO,KK}
and the ``strictly'' part is due to Le~Boudec \cite{Boudec}.
The former showed that the unique trace property is equivalent to having a trivial amenable radical,
a property already known to be weaker than $C^*$-simplicity, see~\cite{Harpe2}.
Both of the works \cite{BKKO,KK} use extensively the theory of boundary actions developed by Furstenberg and Hamana.
Even more recently,
some other more operator-theoretical characterizations have been obtained by Haagerup \cite{Haagerup} and Kennedy \cite{Kennedy}.
However, the conditions from \cite{Haagerup,Kennedy} that are equivalent with $C^*$-simplicity,
are not always easy to check in concrete situations, for example by combinatorial group properties.
We also note that the unique trace property for $G$ always implies that $G$ is icc,
that is, every nontrivial conjugacy class in $G$ is infinite,
and it is well-known that the group von~Neumann algebra associated to $G$ is a factor if and only if $G$ is icc.

The problem of finding conditions to ensure that a free product of groups with amalgamation is $C^*$-simple
was first considered by B{\'e}dos \cite{Bedos},
and mentioned as Problem~27 in \cite{Harpe2}.

In this article, we first make a detailed study of amalgamated free products,
inspired by work of de~la~Harpe and Pr{\'e}aux \cite{HP}.
By making use of a few new observations, we are able to improve some of the results in \cite{HP}.
Then we present in Section~\ref{sec:example} a concrete example of an amalgam that has the unique trace property, but is not $C^*$-simple.

In particular, our example shows that a free product with amalgamation can fail to be $C^*$-simple when it has a trivial kernel,
but has ``one-sided'' kernels that are nontrivial (and amenable).
Every amalgam $G_0*_HG_1$ acts on its Bass-Serre tree, and in this setup the one-sided kernels
consist of the elements that fix all vertices of the ``half-trees'' obtained by removing the edge corresponding to $H$.
This has similarities with \cite{Boudec}.

We keep the first sections mostly to ``elementary'' proofs,
restricting to combinatorial group properties,
and delay the involvement of boundary actions to Section~\ref{sec:actions}.

In the final two sections, we first recall the definitions of radical and residual classes of groups
and prove several statements that will be of later use.
The result of \cite{BKKO} implies that
the class of groups with the unique trace property is the residual class ``dual'' to the radical class of amenable groups.
We then show that the class of $C^*$-simple groups is also a residual class,
giving rise to a ``predual'' radical class of groups that contains all the amenable groups,
and we will call a group ``amenablish'' if it belongs to this class.
The example we provide of an amalgam that is not $C^*$-simple, but has the unique trace property,
turns out to be a (nonamenable) amenablish group.
Moreover,
for every group there is an amenablish radical such that $C^*$-simplicity of the group is equivalent to this radical being trivial,
giving an analog of the relation between the amenable radical and the unique trace property.

\section{Preliminaries}

We consider only discrete groups in this article.

Let $G$ be a group.
As usual, we equip the Hilbert space $\ell^2(G)$ with the standard orthonormal basis $\{\delta_g\}_{g\in G}$,
and define the left regular representation $\lambda$ of $G$ on $\ell^2(G)$ by $\lambda(g)\delta_h=\delta_{gh}$.
Then the reduced group $C^*$-algebra of $G$, denoted by $C^*_r(G)$, is the $C^*$-subalgebra of $B(\ell^2(G))$ generated by $\lambda(G)$.
The group $G$ is called \emph{$C^*$-simple} if $C^*_r(G)$ is simple, that is, if it has no nontrivial proper two-sided closed ideals.

A state on a unital $C^*$-algebra $A$ is a linear functional $\phi\colon A\to\C$ that is positive,
i.e., $\phi(a)\geq 0$ whenever $a\in A$ and $a\geq 0$, and unital, i.e., $\phi(1)=1$.
A state $\phi$ is called tracial if it satisfies the additional property that $\phi(ab)=\phi(ba)$ for all $a,b\in A$.
There is a canonical faithful tracial state $\tau$ on $C^*_r(G)$,
namely the vector state associated with $\delta_e$,
that is, $\tau(a)=\langle a\delta_e,\delta_e\rangle$ for all $a\in C^*_r(G)$.
The group $G$ is said to have the \emph{unique trace property} if $\tau$ is the only tracial state on $C^*_r(G)$.

Recall that a group $G$ is \emph{amenable} if there exists a state on $\ell^\infty(G)$
which is invariant under the left translation action by $G$.
It is explained by Day \cite{Day} that every group $G$ has a unique maximal normal amenable subgroup,
called the \emph{amenable radical} of $G$.
Then \cite[Theorem~1.3]{BKKO} shows that $G$ has the unique trace property if and only if the amenable radical of $G$ is trivial,
so $C^*$-simplicity is stronger than the unique trace property by \cite[Proposition~3]{Harpe2}.

The first larger class of groups that were shown to be $C^*$-simple with the unique trace property,
was the so-called Powers groups introduced by de~la~Harpe \cite{Harpe} (see also \cite{Harpe-Jhabvala}).

\begin{definition}
A group $G$ is called a \emph{Powers group} if for any finite subset $F$ of $G\setminus\{e\}$ and any integer $k\geq 1$
there exists a partition $G=D\sqcup E$ and elements $g_1,g_2,\dotsc,g_k$ in $G$ such that
\[
\begin{split}
fD \cap D &= \varnothing \text{ for all } f\in F,\\
g_iE \cap g_jE &= \varnothing \text{ for all distinct } 1\leq i,j \leq k.
\end{split}
\]
Moreover, a group $G$ is called a \emph{weak Powers group}
if it satisfies the above definition for all finite sets $F$ that are contained in a nontrivial conjugacy class.
Finally, $G$ is called a \emph{weak$^*$~Powers} group if it satisfies the above definition for all nontrivial one-element sets $F$.
\end{definition}

Clearly, every Powers group is a weak Powers group and every weak Powers group is a weak$^*$~Powers group.
It is known that the weak Powers property implies $C^*$-simplicity \cite{Boca-Nitica},
while the weak$^*$~Powers property implies the unique trace property \cite[Section~5.2]{TD}.

\begin{lemma}\label{lem:normal Powers}
Assume that $N$ is a normal subgroup of a group $G$.
If $G$ is a Powers group, then $N$ is a Powers group.
If $G$ is a weak Powers group, then $N$ is a weak Powers group.
If $G$ is a weak$^*$~Powers group, then $N$ is a weak$^*$~Powers group.
\end{lemma}

\begin{proof}
The first statement is a (seemingly unnoticed) result by Kim \cite[Theorem~1]{Kim},
and the last two statements are identically proven.
For the convenience of the reader, we present the argument of the middle one:

Let $k\geq 1$, and suppose that $F$ is a finite set and $C'$ a nontrivial conjugacy class in $N$ such that $F\subseteq C'$.
If two elements are conjugate in $N$ they are also conjugate in $G$,
so there is a nontrivial conjugacy class $C$ in $G$ such that $F\subseteq C$.
Since $G$ is weak Powers, there exists a partition $G=D\sqcup E$ and $g_1,g_2,\dotsc, g_k$ such that
$fD\cap D=\varnothing$ for all $f\in F$ and $g_iE\cap g_jE=\varnothing$ whenever $i\neq j$.

If $k=1$, set $D'=\{e\}$, $E'=N\setminus\{e\}$, and $s_1=e$.

If $k\geq 2$, set $D'=D\cap N$ and $E'=E\cap N$,
so $N=D'\sqcup E'$ and $fD'\cap D'\subseteq fD\cap D=\varnothing$ for all $f\in F$.

Furthermore, fix an $f\in F$ and note that the sets $g_1^{-1}g_iE$ are mutually disjoint for $i\geq 1$,
so that in particular, for $i\geq 2$, then $g_1^{-1}g_iE\cap E=\varnothing$, i.e., $g_1^{-1}g_iE\subseteq G\setminus E=D$.
This again implies that $fg_1^{-1}g_iE\cap D\subseteq fD\cap D=\varnothing$, so $fg_1^{-1}g_iE\subseteq G\setminus D=E$.

Set $s_1=e$ and $s_i=g_1^{-1}g_ifg_i^{-1}g_1$ for $i\geq 2$.
Since $N$ is normal in $G$, we have $s_i\in N$ for all $i$, and if $i\neq j$, then
\[
\begin{split}
s_iE'\cap s_jE'&=g_1^{-1}g_ifg_i^{-1}g_1E'\cap g_1^{-1}g_jfg_j^{-1}g_1E'
\subseteq g_1^{-1}g_ifg_i^{-1}g_1E \cap g_1^{-1}g_jfg_j^{-1}g_1E \\
& \subseteq g_1^{-1}g_iE \cap g_1^{-1}g_jE = \varnothing.
\end{split}
\]
Kim also proves that $D'$ and $E'$ are nonempty, but that does not seem to be necessary.
\end{proof}

In \cite{HP} a group is said to be ``strongly Powers'' if every subnormal subgroup is a Powers group.
A consequence of the above lemma is that the notion of strongly Powers coincides with Powers.

\medskip

We conclude this section by mentioning two definitions.
First, a group $G$ has the \emph{free semigroup property} if for any finite subset $F$ of $G$ there exists $g$ in $G$
such that $gF$ is semifree, that is, the subsemigroup generated by $gF$ in $G$ is free over $gF$.
Finally, a group $G$ is said to have \emph{stable rank one} if its reduced group $C^*$-algebra has stable rank one, that is,
the invertible elements of $C^*_r(G)$ are dense in $C^*_r(G)$.

\section{Free products of groups with amalgamation}

Recall that (see e.g.\ \cite{Serre}) a free product of groups $G_0$ and $G_1$ with amalgamation over a common subgroup $H$
(embedded via injective homomorphisms $H \to G_i$ for $i=0,1$)
is a group $G$ together with homomorphisms $\phi_i \colon G_i \to G$ for $i=0,1$, that agree on $H$,
universal in the sense that for any other group $G'$ with homomorphisms $\phi_i' \colon G_i \to G'$ that agree on $H$,
there is a unique homomorphism $\phi \colon G \to G'$ such that $\phi'_i = \phi \circ \phi_i$.

The amalgamated free products $G = G_0 *_H G_1$ that we consider in this article
are always assumed to be nondegenerate in the sense that $([G_0 : H]-1) \cdot ([G_1 : H]-1) \geq 2$,
otherwise the situation is very different.

Let $\ker G = \bigcap_{g \in G}gHg^{-1}$ denote the kernel of the amalgam.
It coincides with the normal core of $H$ in $G$,
i.e., the largest subgroup of $H$ which is normal in $G$ (and thus in $G_0$ and $G_1$).

If $N$ is any subgroup of $H$ which is normal in $G$,
then the quotient $G/N$ is isomorphic to $(G_0/N)*_{H/N}(G_1/N)$.
In particular, this holds when $N=\ker G$.
Note that we always have $\ker(G/\ker G)=\{e\}$.
Indeed, if $\ker(G/\ker G)=N\subseteq H/\ker G$, let $N'$ be the inverse image of $N$ under the quotient map $G\to G/\ker G$.
Then $N'$ is a normal subgroup of $G$ containing $\ker G$ and sitting inside $H$,
so we must have $N'=\ker G$ by maximality of $\ker G$.

If $G$ is nondegenerate, then $G/\ker G$ is nondegenerate, as $[G_i : H]=[G_i/\ker G : H/\ker G]$.

Let $FC(G)$ denote the normal subgroup of $G$ consisting of elements with finite conjugacy class in $G$.
Clearly, if $g \in G_i \setminus H$, then the conjugacy class of $g$ in $G$ must be infinite.
Hence, $FC(G)$ is contained in $H$.

Let now $AR(G)$ denote the amenable radical of $G$
and let $NF(G)$ denote the largest normal subgroup of $G$ that does not contain any nonabelian free subgroup.
By \cite[Proposition~7]{Cornulier}, these groups fit in general into a sequence of subgroups, namely
\begin{equation}\label{eq:chain}
FC(G) \subseteq AR(G) \subseteq NF(G) \subseteq \ker G \subseteq H.
\end{equation}
In particular, $NF(G)\subsetneq G$, so $G$ always contains a nonabelian free group.
Moreover, there is an even longer sequence
\[
FC(G) \subseteq FC(\operatorname{ker}G) \subseteq AR(\ker G) = AR(G) \subseteq NF(G) = NF(\ker G) \subseteq \ker G.
\]
The first containment is clear, the second holds since $FC(\ker G)$ is an amenable normal subgroup of $\ker G$,
and the two equalities follow from Examples~\ref{ex:AR}, \ref{ex:NF}, and Lemma~\ref{lem:normal radical}.

If $H$ is finite, then $FC(G)=\ker G$, so the sequence collapses.
Indeed, if $h\in\ker G$, then $h \in gHg^{-1}$ for all $g \in G$, so $g^{-1}hg \in H$ for all $g \in G$.
Hence, the conjugacy class of $h$ is contained in $H$, which is finite, so $h \in FC(G)$.
See Theorem~\ref{thm:finite amalgam} for more on this case.

\begin{proposition}\label{prop:amalgam ut}
A nondegenerate amalgamated free product $G$ has the unique trace property if and only if $\ker G$ has the unique trace property.
\end{proposition}

\begin{proof}
Since $AR(G)=AR(\ker G)$ as explained above, this follows from \cite[Theorem~1.3]{BKKO}.
\end{proof}

Let us now recall the normal form for an element in $G$.
First, for $i=0,1$, we choose sets $S_i$ so that $S_i\cup\{e\}$ form systems of representatives for the left cosets of $G_i/H$.
Then every element of $G$ can be written uniquely as either $s_0s_1 \cdots s_{n-1}s_nh$ or $s_1 \cdots s_{n-1}s_nh$,
where $s_i\in S_{i\pmod 2}$ and $h\in H$.
To avoid division into separate cases,
the notation $(s_0)s_1s_2\dotsm s_{n-1}s_nh$ for the normal form of an element of $G$ is often used,
especially in Section~\ref{sec:actions}.

Since we will not always assume that elements are on normal form,
the following observation is sometimes useful:
for $i=0,1$, if $g_i\in G_i\setminus H$ and $h\in H$,
then there exists $g_i'\in G_i\setminus H$ such that $g_ih=hg_i'$.
In other words, we can ``cycle'' a letter from $H$ through a word $(g_0)g_1\dotsm g_n$ in $G$,
with $g_i\in G_{i\pmod 2}\setminus H$, without changing the length of the word.

Let $G$ be any nondegenerate free product with amalgamation.
We decompose $G$ as follows.
For $j=0,1$ and $k\geq 1$, let
\[
T_{j,k}=\{g_0\dotsm g_{k-1} : g_i\in G_{i+j\pmod 2}\setminus H\},
\]
i.e., $T_{j,k}$ consists of all words of length $k$ starting with a letter from $G_j\setminus H$.
To simplify notation, we set $T_{0,0}=T_{1,0}=H$.
For $j=0,1$ and every $k\geq 0$, we now define the set
\begin{equation}\label{eq:C-jk}
C_{j,k}=\bigcap_{g\in T_{j,k}}gHg^{-1},
\end{equation}
and note that $C_{0,0}=C_{1,0}=H$ and that $H\cap C_{j,k}$ is always a normal subgroup of $H$.
Then
\[
\ker G=\bigcap_{\substack{k\geq 0 \\ j=0,1}} C_{j,k},
\]
which, as mentioned above, is always normal in $G$.
Next, we set
\begin{equation}\label{eq:K0K1}
K_0=\bigcap_{k\geq 0}C_{0,k}
\quad\text{and}\quad
K_1=\bigcap_{k\geq 0}C_{1,k}.
\end{equation}
Both $K_0$ and $K_1$ are normal subgroups of $H$.
Remark that it actually follows that
\begin{equation}\label{eq:even-length}
K_0=\bigcap_{k\geq 0}C_{0,2k}
\quad\text{and}\quad
K_1=\bigcap_{k\geq 0}C_{1,2k}.
\end{equation}
Indeed, if $h\notin K_0$, then $g^{-1}hg\notin H$ for some $g\in T_{0,n}$ with $n\geq 1$.
If $n$ is odd, take any $g_1\in G_1\setminus H$, and note that $gg_1\in T_{0,n+1}$ and $g_1^{-1}g^{-1}hgg_1\notin H$,
so $h\notin gg_1Hg_1^{-1}g^{-1}$, hence, $h\notin \bigcap_{k\geq 0}C_{0,2k}$.
A similar argument works for $K_1$.

Moreover, it should be clear that $\ker G=K_0\cap K_1$ and that
\begin{equation}\label{eq:K0-K1-relation}
K_0=H\cap\bigcap_{g_0\in G_0\setminus H}g_0K_1g_0^{-1}
\quad\text{and}\quad
K_1=H\cap\bigcap_{g_1\in G_1\setminus H}g_1K_0g_1^{-1}.
\end{equation}
If $K_0=\ker G$, then $K_1\subseteq\ker G$ since $\ker G$ is normal, so $K_0=K_1$ (similarly if $K_1=\ker G$).
Also, we have that $K_0=\{e\}$ if and only if $K_1=\{e\}$.

Finally, using the notation from \cite[(i)~p.~2-3]{HP} we set
\begin{equation}\label{eq:finite-length}
C_k=\bigcap_{\substack{0\leq n\leq k \\ j=0,1}} C_{j,n}.
\end{equation}
Note that for any $g\in T_{0,k+1}$, we have $K_0\subseteq gC_kg^{-1}$,
and for any $g\in T_{1,k+1}$, we have $K_1\subseteq gC_kg^{-1}$.
Thus, if $C_k$ is trivial for some $k\geq 0$, then $K_0=K_1=\{e\}$.
Indeed, pick $k\geq 0$, $g\in T_{0,k+1}$, $h\in K_0$, and let $s$ be an arbitrary element of length $\leq k$.
Then $gs\in T_{0,n}$ for some $n\geq 1$, so $h\in gsHs^{-1}g^{-1}$.
Since, this holds for all such $s$, we have $h\in gC_kg^{-1}$.

Moreover, it is worth noticing the difference between $K_0$, $K_1$, and $C_k$:
\begin{itemize}
\item the intersections defining the $C_k$'s involve conjugation of elements of finite length (bounded by some $k$),
while the first letter of the words can come from any of $G_0,G_1$.
\item the intersections defining the $K_i$'s involve conjugation of elements of arbitrary length,
where the first letter of the words comes from the same group.
\end{itemize}

\begin{theorem}\label{thm:trivial-Ks}
Let $G=G_0*_H G_1$ be a nondegenerate free product with amalgamation, and let $K_0$, $K_1$ be as defined above.
Then the following are equivalent:
\begin{itemize}
\item[(i)] $K_0=K_1=\{e\}$.
\item[(ii)] for every finite $F\subseteq G\setminus\{e\}$, there exists $g\in G$ such that $gFg^{-1}\cap H=\varnothing$.
\item[(iii)] for every finite $F\subseteq H\setminus\{e\}$, there exists $g\in G$ such that $gFg^{-1}\cap H=\varnothing$.
\end{itemize}
Moreover, any one of these equivalent conditions implies that $G$ is a Powers group and has the free semigroup property.
\end{theorem}

\begin{proof}
It is obvious that (ii) implies (iii).

To see that (iii) implies (i) suppose first that $\ker G\neq \{e\}$.
Then pick $f\in\ker G\setminus\{e\}$ and set $F=\{f\}$ and clearly $gfg^{-1}\in H$ for all $g\in G$,
i.e., $gfg^{-1}\in gFg^{-1}\cap H$.
Next, assume that $\ker G=\{e\}$, but $K_0\neq K_1$.
Recall from \eqref{eq:K0-K1-relation} and the subsequent note that this means both $K_0\neq\{e\}$ and $K_1\neq\{e\}$.
Thus we pick $f_i\in K_i\setminus\{e\}$ for $i=0,1$ and set $F=\{f_0,f_1\}$.
Let $g\in G$ be arbitrary.
If $g\in H$ there is nothing to show, so assume that $g$ ends with a letter from $G_i\setminus H$.
But then $gf_ig^{-1}\in H$, i.e., $gf_ig^{-1}\in gFg^{-1}\cap H$.

Finally we prove that (i) implies (ii).
Choose an arbitrary finite set $F\subseteq G\setminus\{e\}$.
The idea is to conjugate the elements from $F$ out of $H$ one by one,
and at the same time make sure we do not conjugate any elements back into $H$.

Assume first there is an element $f_1\in F \cap H$ (else there is nothing to show).
Because $K_0$ is trivial, then \eqref{eq:even-length} means that all elements in $H$
can be conjugated out of $H$ by an element of even length starting with a letter from $G_0\setminus H$.
That is, we can find $r_1=g_0g_1\dotsm g_{2n_1-1}$ such that $g_i\in G_{i\pmod 2}\setminus H$ and $r_1^{-1}f_1r_1\notin H$.

Next, consider the set $F_1=\{r_1^{-1}fr_1 : f\in F\}$.
Assume that there is an element $f_2\in F$, so that $r_1^{-1}f_2r_1\in H$ (otherwise we are done).
Then there exists $r_2=g_0'g_1'\dotsm g_{2n_2-1}'$ such that $g_i'\in G_{i\pmod 2}\setminus H$ and $r_2^{-1}r_1^{-1}f_2r_1r_2\notin H$.
This also means that $r_2^{-1}r_1^{-1}f_1r_1r_2\notin H$.
Indeed, let $j$ be the smallest number such that $g_j^{-1}\dotsm g_0^{-1}f_1g_0\dotsm g_j\notin H$, i.e.,
it belongs to $G_{j\pmod 2}\setminus H=T_{j\pmod 2,1}$.
Then $g_{j+1}^{-1}g_j^{-1}\dotsm g_0^{-1}fg_0\dotsm g_jg_{j+1}\in T_{j+1\pmod 2,3}$,
and as we continue to conjugate by elements alternating between $G_0\setminus H$ and $G_1\setminus H$ this product will only increase in length.

Now we set $F_2=\{r_2^{-1}r_1^{-1}fr_1r_2 : f\in F\}$.
If all elements of $F_2$ are outside $H$ we are done, so assume that there is some $f_3\in F$ such that $r_2^{-1}r_1^{-1}f_3r_1r_2\in H$.

It should be clear how this process continues,
and since $F$ is finite, we take $r$ to be the product of the $r_i$'s,
and then $r^{-1}fr\notin H$ for every $f\in F$.

The last two observations follow from \cite[Proposition~10]{Harpe} and \cite[Example~4.4~(iii),(iv), and Remark~4.5]{DH}.
\end{proof}

\begin{remark}
The above result shows that \cite[Theorem~3~(i)]{HP} can be slightly improved,
as Lemma~\ref{lem:normal Powers} shows that ``strongly Powers'' is the same as Powers.
In fact, using the comment following \eqref{eq:finite-length},
one can replace ``$C_k=\{e\}$ for some $k$'' with any of the equivalent properties of Theorem~\ref{thm:trivial-Ks}.
Additionally, the countability assumption is no longer needed.

We will come back to the geometric interpretation of $K_0$ and $K_1$ in Section~\ref{sec:actions}.
In particular, Proposition~\ref{prop:slender} below gives more properties equivalent to those in Theorem~\ref{thm:trivial-Ks}.
\end{remark}

\begin{remark}
In Section~\ref{sec:example}, we give an explicit example of a group $\Gamma$ for which $\ker\Gamma=\{e\}$,
while $K_0$ and $K_1$ are both nontrivial.
We show that the group has the unique trace property, but is not $C^*$-simple.
This gives a counterexample to the first author's statement \cite[Corollary~4.7]{Ivanov},
which the second author noticed to be incorrect.  
\end{remark}

\begin{proposition}\label{prop:cstarsimple-amalgam}
If $G$ is countable, $H$ is amenable, and $K_0$ or $K_1$ is nontrivial, then $G$ is not $C^*$-simple.
\end{proposition}

This result is generalized in Theorem~\ref{thm:K0K1-amenable} below,
but in the context of recent works, we provide two other short proofs of the statement.

\begin{proof}[First proof]
We will show that $H$ is recurrent in $G$ in the sense of \cite[Definition~5.1]{Kennedy}.
Let $(g_n)$ be a sequence in $G$.
Then at least one of the following holds:
\begin{itemize}
\item[(a)] infinitely many elements from $(g_n)$ belong to $H$
\item[(b)] infinitely many elements from $(g_n)$ start with a letter from $G_0\setminus H$
\item[(c)] infinitely many elements from $(g_n)$ start with a letter from $G_1\setminus H$
\end{itemize}
Pick a subsequence $(g_{n_k})$ of $(g_n)$ with elements from the one of (a),(b),(c) that holds.
If it is (a) then $H=\bigcap_k g_{n_k}Hg_{n_k}^{-1}$,
if it is (b) then $K_0\subseteq\bigcap_k g_{n_k}Hg_{n_k}^{-1}$,
and if it is (c) then $K_1\subseteq\bigcap_kg_{n_k}Hg_{n_k}^{-1}$.
If one of $K_0$ and $K_1$ is nontrivial, then both the other one and $H$ are nontrivial as well. 
Hence, it follows from \cite[Theorem~1.1]{Kennedy} that $G$ is not $C^*$-simple.
\end{proof}

\begin{proof}[Second proof]
If $K_0=\ker G$, then $K_1=\ker G$ by the comment following \eqref{eq:K0-K1-relation}.
Thus, by assumption, $\ker G$ is a nontrivial normal amenable subgroup of $G$, and hence $G$ cannot be $C^*$-simple.
The similar argument holds if $K_1=\ker G$, so we may assume that both $K_0$ and $K_1$ are different from $\ker G$.

Choose $a\in K_0\setminus\ker G$ and $b\in K_1\setminus\ker G$.
Then
\[
\begin{split}
\{gH : gH\neq agH\}&\subseteq\{gH:g\in T_{1,k}\text{ for some }k\geq 1\} \quad\text{and}\\
\{gH : gH\neq bgH\}&\subseteq\{gH:g\in T_{0,k}\text{ for some }k\geq 1\},
\end{split}
\]
which are clearly disjoint.
By using the technique from \cite[Proposition~5.8]{Haagerup-Olesen}, explained in \cite[p.~12]{Ozawa},
the action of $G$ on $G/H$ gives rise to a unitary representation $\pi\colon G\to\ell^2(G/H)$,
that extends to a continuous representation of $C^*_r(G)$.
It follows that $(1-\lambda(a))(1-\lambda(b))$ generates a proper two-sided closed ideal of $C^*_r(G)$.
Hence, $G$ is not $C^*$-simple.
\end{proof}

\begin{example}
For any triple of groups $H$, $G_0$, and $G_1$, we have
\[
G = (G_0 \times H) *_H (H \times G_1) \cong (G_0 * G_1) \times H.
\]
In this case $H=\ker G=K_0=K_1$, and $G$ is $C^*$-simple if and only if $H$ is $C^*$-simple.
In particular, this means that $G$ can be $C^*$-simple even if $\operatorname{ker}G$ is nontrivial.
\end{example}

We now consider the special case where $H$ is finite.

\begin{theorem}\label{thm:finite amalgam}
Let $G=G_0*_H G_1$ be a nondegenerate free product with amalgamation,
and assume that $H$ is finite.

Then the following are equivalent:
\begin{itemize}
\item[(i)] $G$ is icc
\item[(ii)] $\operatorname{ker}G=\{e\}$
\item[(iii)] $K_0=K_1=\{e\}$
\item[(iv)] $G$ is Powers
\item[(v)] $G$ is $C^*$-simple
\item[(vi)] $G$ has the unique trace property
\item[(vii)] there exists $g \in G$ such that $H \cap gHg^{-1} = \{e\}$
\item[(viii)] $G$ has the free semigroup property
\end{itemize}
\end{theorem}

\begin{proof}
(i)~$\Longrightarrow$~(ii):
Since $\ker G\subseteq H$, it is a finite normal subgroup of $G$, so it must be trivial when $G$ is icc.

(ii)~$\Longrightarrow$~(iii):
Let $C_k$ be defined as in \eqref{eq:finite-length},
so that $\ker G$ is the intersection of the decreasing chain $C_0\supseteq C_1\supseteq C_2\supseteq \dotsb$.
Since all the $C_k$'s are subgroups of $H$, they are finite, so if $\ker G=\{e\}$,
we must have $C_k=\{e\}$ for some $k$.
Hence, the remark following \eqref{eq:finite-length} explains that $K_0=K_1=\{e\}$.

(iii)~$\Longrightarrow$~(iv) follows from Theorem~\ref{thm:trivial-Ks}.

(iv)~$\Longrightarrow$~(v)~$\Longrightarrow$~(vi)~$\Longrightarrow$~(i) is known to hold for all groups.

(iii)~$\Longrightarrow$~(vii) is also a consequence of Theorem~\ref{thm:trivial-Ks}, by taking $F=H\setminus\{e\}$.

(vii)~$\Longrightarrow$~(viii) follows from \cite[Proposition~5.1]{DH}.

(viii)~$\Longrightarrow$~(ii):
If $H$ is finite and $\ker G\neq\{e\}$,
then there is no $g\in G$ such that $g\ker G$ is semifree, i.e., $G$ does not have the free semigroup property.
\end{proof}

\begin{remark}
If $G$ is an amalgam with finite $H$ satisfying the equivalent conditions of Theorem~\ref{thm:finite amalgam},
then $G$ has stable rank one by \cite[Theorem~1.6]{DH}
(where the argument is based on \cite[Corollary~3.9]{DHR} for trivial $H$).
However, there also exists amalgams $G$ with finite $H$ such that $G$ has stable rank one,
but $G$ does not satisfy the conditions of Theorem~\ref{thm:finite amalgam},
as explained in the paragraph following \cite[Theorem~1.6]{DH}.

Moreover, it follows from condition~(vii) and \cite[Corollary~3.6]{Ivanov} that for $G = G_0 *_H G_1$,
the positive cone of the ${\bf K_0}$-group of $C^*_r(G)$ is not perforated, i.e.,

\[
{\bf K_0}(C^*_r(G))^+ = \{ \gamma \in {\bf K_0}(C^*_r(G)) : {\bf K_0}(\tau)(\gamma) > 0 \} \cup \{ 0 \},
\]
where $\tau$ is the canonical tracial state on $C^*_r(G)$.
Indeed, the assumption ${\bf K_1}(C^*_r(H))=0$ is easily seen to hold from the fact that for a finite group $H$,
$C^*_r(H)$ is a direct sum of matrix algebras and those have trivial ${\bf K_1}$-groups.
\end{remark}

\section{An example}\label{sec:example}

In this section $\oplus$, when applied to indices, will denote summation modulo $2$, and $\N$ will denote the positive integers.

The group $\Gamma = G_0 *_H G_1$ is defined as follows: first $H$ is given by the set of generators
\[
\{ h(i_1,i_2,\dotsc,i_n) : \text{$n\in\N$ and $i_k\in\{0,1\}$ for all $k\in\{1,\dotsc,n\}$}\},
\]
with the relations
\[
h(i_1,i_2,\dotsc,i_n)^2 = e \, \text{ for all $n\in\N$, $k\in\{1,\dotsc,n\}$, and $i_k \in \{0,1\}$},
\]
and for $n\geq k$
\begin{multline*}
h(i_1,\dotsc,i_k) h(j_1,\dotsc,j_n) h(i_1,\dotsc,i_k) = \\
\begin{cases}
h(j_1, \dotsc, j_k, j_{k+1} \oplus 1, \dotsc, j_n) & \text{if $n>k$ and $i_\ell = j_\ell$ for all $1\leq\ell\leq k$}, \\
h(j_1, \dotsc, j_k, j_{k+1}, \dotsc, j_n) & \text{otherwise}.
\end{cases}
\end{multline*}
Then define
\[
\Gamma=\langle H\cup\{ g_0, g_1 \}\rangle,
\]
with the additional relations
\begin{gather*}
g_0^2 = g_1^2 = e,\quad (g_0 h(1))^3 = e, \quad  (g_1 h(0))^3 = e, \\
g_0 h(1, 0, i_3, \dotsc, i_n) g_0 = h(0, i_3, \dotsc, i_n),
\quad 
g_0 h(1, 1, i_3, \dotsc, i_n) g_0 = h(1, 1, i_3, \dotsc, i_n), \\
g_1 h(0, 0, i_3, \dotsc, i_n) g_1 = h(0, 0, i_3, \dotsc, i_n),
\quad
g_1 h(0, 1, i_3, \dotsc, i_n) g_1 = h(1, i_3, \dotsc, i_n).
\end{gather*}
Finally, let
\[
G_0=\langle H\cup\{ g_0 \}\rangle
\quad\text{and}\quad
G_1=\langle H\cup\{ g_1 \}\rangle,
\]
so that $\Gamma = G_0 *_H G_1$.
Note also that
\[
H = \langle \{h(\underbrace{0,\dotsc,0}_{n \text{ times}}), n\in\N \}
\cup
\{h(1,\underbrace{0,\dotsc,0}_{n-1 \text{ times}}),n\in\N \} \rangle.
\]
\begin{lemma}
For every $g\in G_0\setminus H$, there exists $h\in H$ such that either $g=g_0h$ or $g=h(1)g_0h$,
and for every $g\in G_1\setminus H$, there exists $h\in H$ such that either $g=g_1h$ or $g=h(0)g_1h$.

Consequently, $[G_0 : H] = [G_1 : H] = 3$, and the sets
\begin{equation}\label{eq:coset-rep}
S_0=\{g_0,h(1)g_0\} \quad\text{and}\quad S_1=\{g_1,h(0)g_1\}.
\end{equation}
provide representatives for the nontrivial left cosets of $G_0/H$ and $G_1/H$, respectively.
\end{lemma}

\begin{proof}
First, if $h$ is one of the generators of $H$ and $h\neq h(1)$,
then the defining relations above show that there exists an $h'$ such that $hg_0=g_0h'$.
Moreover, for $h\in H$, with $h\neq h(1)$, we have $hh(1)=h(1)\cdot h(1)hh(1)=h(1)h'$ for some $h'\neq h(1)$,
that is, we can move the $h(1)$'s to the left in the product.
Hence, for any element $g=(g_0)h_1g_0h_2g_0\dotsm g_0h_n(g_0)\in G_0$, where $h_i\in H$,
we can combine the two observations above to find the required $h$.

A similar argument holds in $G_1$.
\end{proof}

Let us now denote the subgroups of $H$ with elements determined by a prescribed start by
\[
H(j_1,\dotsc,j_\ell)
= \langle \{ h(j_1,\dotsc,j_\ell,i_{\ell+1},\dotsc,i_n) : n\in\N, n\geq\ell, i_{\ell+1},\dotsc,i_n \in \{ 0,1 \} \} \rangle
\]
and the subgroups of $H(j_1,\dotsc,j_\ell)$ of elements having arguments of minimum length $k$ by
\[
H_k(j_1, \dotsc, j_\ell)
= \langle \{ h(j_1,\dotsc, j_\ell, i_{\ell+1},\dotsc,i_n) : n \in\N, n\geq k\geq\ell, i_{\ell+1},\dotsc,i_n \in \{ 0,1 \} \} \rangle.
\]

\begin{remark}\label{rem:Gammarelations}
Let $i,j$ be two sequences such that $h(i)$ and $h(j)$ commute.
For any $k,\ell$ greater than the respective lengths of $i,j$, we have
\[
\langle H_k(i)\cup H_\ell(j)\rangle = H_k(i)H_\ell(j)=H_\ell(j)H_k(i).
\]
Let $k\geq 1$ and note that $g_0h(1)g_0 \notin H$.
By using the defining relations for $\Gamma$, we compute
\begin{gather*}
g_0H_k(0)g_0=H_{k+1}(1,0) \\
h(1)g_0H_k(0)g_0h(1)=h(1)H_{k+1}(1,0)h(1)=H_{k+1}(1,1) \\
H\cap g_0H(1)g_0=H(0)H(1,1) \\
H\cap h(1)g_0H(1)g_0h(1)=h(1)(H\cap g_0H(1)g_0)h(1)=h(1)H(0)H(1,1)h(1)=H(0)H(1,0),
\end{gather*}
and then
\begin{gather*}
H\cap g_0Hg_0=H(0)H_2(1) \\
H\cap h(1)g_0Hg_0h(1)=h(1)(H\cap g_0Hg_0)h(1)=h(1)H(0)H_2(1)h(1)=H(0)H_2(1).
\end{gather*}
Finally,
\[
H\cap g_0H_k(0)H(1)g_0=g_0H_k(0)g_0\big(H\cap g_0H(1)g_0\big)=H_{k+1}(1,0)H(0)H(1,1),
\]
where the first equality follows from Dedekind's modular law for groups, and
\[
\begin{split}
H\cap h(1)g_0H_k(0)H(1)g_0h(1)&=h(1)(H\cap g_0H_k(0)H(1)g_0)h(1)\\
&=h(1)H_{k+1}(1,0)H(0)H(1,1)h(1) =H_{k+1}(1,1)H(0)H(1,0).
\end{split}
\]
Similarly,
\begin{gather*}
H\cap g_1Hg_1=H\cap h(0)g_1Hg_1h(0)=H(1)H_2(0) \\
H\cap g_1H(0)H_k(1)g_1=H_{k+1}(0,1)H(1)H(0,0) \\
H\cap h(0)g_1H(0)H_k(1)g_1h(0)=H_{k+1}(0,0)H(1)H(0,1).
\end{gather*}
\end{remark}

\begin{lemma}\label{lem:Gamma-K0K1}
We have $K_0 = H(0)$, $K_1 = H(1)$, and $\ker\Gamma = K_0 \cap K_1 = H(0) \cap H(1) = \{e\}$.
\end{lemma}

\begin{proof}
We use the notation for $C_{j,k}$ from \eqref{eq:C-jk}, so that $C_{0,0}=C_{1,0}=H$,
and remark that for every $j,k$ we have
\[
C_{0,k+1}=\bigcap_{g\in G_0\setminus H}gC_{1,k}g^{-1}
\quad\text{and}\quad
C_{1,k+1}=\bigcap_{g\in G_1\setminus H}gC_{0,k}g^{-1}.
\]
In the following we will make use of Remark~\ref{rem:Gammarelations}, the coset representatives from \eqref{eq:coset-rep},
and the fact that every $C_{j,k}$ is invariant under conjugation by elements of $H$.
We compute that
\[
H \cap C_{0,1}
= H\cap\bigcap_{g\in G_0\setminus H}gHg^{-1}
= H \cap g_0Hg_0 \cap h(1)g_0Hg_0h(1)
= H(0)H_2(1)
\]
and
\[
H \cap C_{1,1}
= H\cap\bigcap_{g\in G_1\setminus H}gHg^{-1}
= H \cap g_1Hg_1 \cap h(0)g_1Hg_1h(0)
= H(1)H_2(0).
\]
Next, let $k\in\N$ and assume $\bigcap\limits_{i=0}^k C_{0,i}=H(0)H_{k+1}(1)$ and $\bigcap\limits_{i=0}^k C_{1,i}=H(1)H_{k+1}(0)$.
Then
\[
\begin{split}
\bigcap_{i=0}^{k+1} C_{0,i}
&= H\cap \bigcap_{i=0}^k \bigcap_{g\in G_0\setminus H} gC_{1,i}g^{-1} \\
&= H\cap g_0\Big(\bigcap_{i=0}^k C_{1,i}\Big)g_0 \cap h(1)g_0\Big(\bigcap_{i=0}^k C_{1,i}\Big)g_0h(1) \\
&= H\cap g_0H(1)H_{k+1}(0)g_0 \cap h(1)g_0H(1)H_{k+1}(0)g_0h(1) \\
&= H_{k+2}(1,0)H(0)H(1,1) \cap H_{k+2}(1,1)H(0)H(1,0) \\
&= H(0)H_{k+2}(1),
\end{split}
\]
and similarly
\[
\bigcap_{i=0}^{k+1} C_{1,i}=H(1)H_{k+2}(0).
\]
Hence, we get that
\[
K_0=\bigcap_{k\geq 0}C_{0,k}=\bigcap_{k\geq 0}H(0)H_{k+1}(1)=H(0)
\]
and
\[
K_1=\bigcap_{k\geq 0}C_{1,k}=\bigcap_{k\geq 0}H(1)H_{k+1}(0)=H(1),
\]
and thus $\ker\Gamma=K_0 \cap K_1=H(0)\cap H(1)=\{e\}$.
\end{proof} 

\begin{theorem}\label{thm:gamma-nonsimple}
The group $\Gamma$ defined above has the unique trace property, but is not $C^*$-simple.
\end{theorem}

\begin{proof}
By Lemma~\ref{lem:Gamma-K0K1} we have $\ker\Gamma=\{e\}$,
so Proposition~\ref{prop:amalgam ut} gives that $\Gamma$ has the unique trace property.
Moreover, $\Gamma$ is countable, $H$ is amenable since it is locally finite, and $K_0$ and $K_1$ are nontrivial.
Hence, it follows from Proposition~\ref{prop:cstarsimple-amalgam} that $\Gamma$ is not $C^*$-simple.
\end{proof}

Note that $\Gamma$ contains uncountably many amenable (abelian) subgroups,
for example the ones generated by subsets of
\[
\{h(\underbrace{0,\dotsc,0}_{n\text{ times}},1) \mid n\in\N)\},
\]
so \cite[Theorem~1.7]{BKKO} does not apply.

We can find an explicit free nonabelian subgroup of $\Gamma$ by
checking that $\langle g_0h(1),g_1h(0) \rangle$ is isomorphic to $\Z_3*\Z_3$,
which clearly contains a nonabelian free group as a subgroup.

Finally, we remark that $\langle K_0 \cup K_1\rangle=H$, and that the normal closure of $H$ is all of $\Gamma$.

\begin{proposition}\label{prop:simple-by-finite}
Let
\[
\theta\colon \Gamma \to \Z/2\Z \times \Z/2\Z
\]
be the group homomorphism defined on generators by
\[
\theta(h(0)) = (-1,1),\quad \theta(h(1)) = (1,-1),\quad \theta(g_0) = (1,-1),\quad \theta(g_1) = (-1,1).
\]
Set $\Gamma'=\ker\theta$.
Then $\Gamma'$ is simple.
\end{proposition}

\begin{proof}
Below we only provide the idea of the argument and omit most of the technicalities.

Observe that since $t^{-1} h(0) t = h(0,i_1, \dotsc, i_{2k})$ for a suitable element $t \in T_{0,2k}$,
it follows that $\theta(h(0,i_1,\dotsc,i_{2k}))=\theta(h(0))=(-1,1)$.
Therefore 
\[
h(0) h(0,i_1, \dotsc, i_{2k}) \in \Gamma'
\]
for all $k\in\N$, $j\in\{1,\dotsc,2k\}$, and $i_j\in\{0,1\}$.
Analogous arguments yield 
\[
h(0) h(1,i_1, \dotsc, i_{2k-1}),\; 
h(1) h(1,i_1, \dotsc, i_{2k}),\;
h(1) h(0,i_1, \dotsc, i_{2k-1})\;
\in\Gamma'
\]
for all $k\in\N$, $j\in\{1,\dotsc,2k\}$, and $i_j\in\{0,1\}$.

Notice also that $g_1h(0),h(0)g_1,g_0h(1),h(1)g_0 \in \Gamma'$.
One can now check that
\begin{multline*} 
\Gamma' = \langle
\{ h(0) h(0,i_1,\dotsc,i_{2k}) , h(0) h(1,i_1,\dotsc,i_{2k-1}), h(1) h(1,i_1,\dotsc,i_{2k}), \\
h(1) h(0,i_1,\dotsc,i_{2k-1}) , \forall k\in\N, \forall i_j\in\{0,1\}, \forall j \} \cup \{ h(0)g_1,h(1)g_0\}
\rangle.
\end{multline*}

Let $N\neq\{e\}$ be a normal subgroup of $\Gamma'$ and pick an element $a \in N\setminus\{e\}$.
The remainder of the proof is about showing that each of the generators of $\Gamma'$ listed above
can be described by a suitable product of conjugates of $a$ by elements of $\Gamma'$.
Since these computations are rather tedious, we leave them out.
\end{proof}

\begin{remark}
If $G$ is an exact group with stable rank one and the unique trace property,
then $G$ is $C^*$-simple by \cite[Theorem~2.1]{B-O}.
Let $\Gamma=G_0*_H G_1$ be the group defined above.
Since $H$ is amenable and has finite index in $G_0$ and $G_1$, both these groups are amenable,
and therefore $\Gamma$ is exact by \cite[Corollary~3.3]{Dykema}.
Hence, $\Gamma$ does \emph{not} have stable rank one by Theorem~\ref{thm:gamma-nonsimple}.
\end{remark}

\begin{remark}\label{rmkGamma}
It was pointed out to us by Adrien Le~Boudec that the group $\Gamma$ is isomorphic to the group $G(A_3,S_3)^{\star}$,
which is one of the examples from \cite[Section~5]{Boudec}.

Let $S_3$ and $A_3$ denote the symmetric and alternating group, respectively, on a three-element set.
Consider a $3$-regular tree $T$ and color the edges $\{1,2,3\}$, so that neighboring edges have different colors.
Let $\sigma(g,v) \in S_3$ be the permutation of the three colors induced in the natural way by the element $g\in\operatorname{Aut}(T)$.
Then $G(A_3,S_3) < \operatorname{Aut}(T)$ is the group of all automorphisms $g$ of $T$ such that $\sigma(g,v) \in A_3$ for all but finitely many vertices $v$.
Then $G(A_3,S_3)^{\star}$ is the subgroup of $G(A_3,S_3)$ of index two preserving the natural bipartition of vertices of $T$.

To see this, remark that by Bass-Serre's theory (cf.\ \cite[I.4 Theorem~6]{Serre}) this index two subgroup is an amalgamated product $\bar{G_0} *_{\bar{H}} \bar{G_1}$,
where $\bar{G_0}$ and $\bar{G_1}$ are the stabilizers of two adjacent vertices (we will call them $v_0$ and $v_1$),
and $\bar{H}$ is the stabilizer of the edge between them (we will call it $e$).
Then, using the notation from \cite[Subsection~3.1]{Boudec2}, we may write $\bar{G_i}$ as an increasing union
\[
G(S_3)_{v_i} = \bigcup_{n\geq 1} K_n(v_i),\quad i=0,1,
\]
where $K_n(v_i)$ is the subgroup of $G(A_3,S_3)$ that fixes the vertex $v_i$
and has $\sigma(g,v) \in A_3$ for all $g \in K_n(v_i)$ and all $v$ that are at a distance larger than $n$ from $v_i$.
With this notation, the element $h(0,i_2, \dotsc, i_n)$ acts on $T$ by swapping the two half-trees,
emanating from a vertex (we will call it $v(0,i_2, \dotsc, i_n)$) that is at a distance $n-1$ from $v_0$ and at a distance $n$ from $v_1$,
and not intersecting the geodesic between $v_0$ and $v(0,i_2, \dotsc, i_n)$.
Clearly, $\sigma(h(0,i_2,\dotsc,i_n), v(0,i_2,\dotsc,i_n)) \notin A_3$, because it leaves one edge (therefore one color) fixed.
The matching of the other vertices of the half-trees is defined so the local permutations belong to $A_3$.
This matching is just a matter of orientation of the tree.
Adding the element $g_1$ to the picture, we see that $K_n(v_0)$ is isomorphic to the wreath product $\Z_2 \wr \dotsb \wr \Z_2 \wr S_3$ ($n-1$~factors of $\Z_2$),
where the top elements beneath the $S_3$ factor are $h(1)$, $h(0,0)$, and $h(0,1)$.
Likewise $K_n(v_1)$ is isomorphic to $\Z_2 \wr \dotsb \wr \Z_2 \wr S_3$ ($n-1$~factors of $\Z_2$),
where the top elements beneath the $S_3$ factor are $h(0)$, $h(1,0)$ and $h(1,1)$.
In this way, we see that $\bar{G_i}$ is isomorphic to $G_i$, $i=0,1$, and therefore $\Gamma$ is isomorphic to $G(A_3,S_3)^{\star}$. 

The simplicity of the group $\Gamma'$ is not covered by \cite[Corollary~4.20]{Boudec2}, because $A_3$ is not generated by its point stabilizers.

We finally note that $G(A_3,S_3)$ is isomorphic to $\Gamma \rtimes \Z_2$,
where $\Z_2$ acts on $\Gamma$ by interchanging the indices $0$ and $1$ of all generating elements of $\Gamma$.
\end{remark}

\section{Actions of free products with amalgamation}\label{sec:actions}

Let $G$ be any group acting on a space $X$, and let us first recall some notation.
The stabilizer subgroup of an element $x\in X$ is $G_x=\{g\in G:gx=x\}$,
and the fixed-point set of $g\in G$ is $X^g=\{x\in X:gx=x\}$.
The \emph{kernel} of the action is the set of elements in $G$ acting trivially on $X$, namely
\[
\ker(G\curvearrowright X)=\bigcap_{x\in X}G_x=\{g\in G:X^g=X\}.
\]
Note that for every $x\in X$ and all $s,g\in G$,
we have $sG_xs^{-1}=G_{sx}$ and $X^{sgs^{-1}}=sX^g$,
so the kernel is a normal subgroup of $G$.
The action is called \emph{faithful} when the kernel is trivial, i.e.,
if for every $g\in G\setminus\{e\}$ there exists $x\in X$ such that $gx\neq x$,
and \emph{strongly faithful} if for every finite set $F\subseteq G\setminus\{e\}$
there exists $x\in X$ such that $gx\neq x$ for all $g\in F$.

Furthermore, the action of $G$ on $X$ is called \emph{free} if whenever $g\in G$, $x\in X$, and $gx=x$, then $g=e$.
Since $\langle G_x : x\in X \rangle$ is invariant under conjugation, it is a normal subgroup of $G$,
which coincides with the subgroup of $G$ generated by $\{g\in G:X^g\neq\varnothing\}$.
This subgroup is the so-called ``join'' of $\{G_x:x\in X\}$, while $\ker(G\curvearrowright X)$ is the ``meet'' of $\{G_x:x\in X\}$.
Obviously, $G$ acts freely on $X$ if and only if this subgroup is trivial.

Let $X$ be a topological space, and suppose that $G$ acts continuously on $X$, that is, by homeomorphisms.
For every $x\in X$ define $G_x^o$ as the subgroup of $G_x$ consisting of all elements that fix a neighborhood of $x$ pointwise.
We notice that $g\in G_x^o$ for some $x\in X$ if and only if $X^g$ has nonempty interior,
and we define the \emph{interior} of the action as
\begin{equation}\label{eq:interior}
\operatorname{int}(G\curvearrowright X)=\langle G_x^o : x\in X\rangle=\langle\{g\in G:X^g\text{ has nonempty interior}\}\rangle.
\end{equation}
Then, by using the identity $X^{sgs^{-1}}=sX^g$,
we see that $\{g\in G:X^g\text{ has nonempty interior}\}$ is invariant under conjugation.
Indeed, if $X^g$ contains a nonempty open subset $V$, then $sV$ is a nonempty open subset of $sX^g=X^{sgs^{-1}}$.
Therefore, $\operatorname{int}(G\curvearrowright X)$ is automatically a normal subgroup of $G$.
One may also check that $sG_x^os^{-1}=G_{sx}^o$ for every $x\in X$ and $s\in G$.
We say that the action is \emph{topologically free} if $X^g$ has empty interior for every $g\in G\setminus\{e\}$,
that is, if the interior is trivial.
Finally, we note that for a topological space $X$, the interior and kernel of an action corresponds to the join and meet,
respectively, of $\{G_x^o:x\in X\}$.

\medskip

Now, fix a nondegenerate amalgam $G=G_0*_H G_1$ and let $T$ denote its Bass-Serre tree
(cf.~\cite{Serre}, see also \cite{HP} and references therein).
Then $T$ has vertex set $G/G_0 \sqcup G/G_1$ and (geometric) edge set $G/H$.
Two vertices in $T$ are adjacent if either of the form
\[
(g_0)g_1\dotsm g_{2n-1} G_0 \xrightarrow{(g_0)g_1\dotsm g_{2n-1}g_{2n}H} (g_0)g_1\dotsm g_{2n-1}g_{2n} G_1
\]
or of the form (for a transversal edge)
\[
(g_0)g_1\dotsm g_{2n} G_1 \xrightarrow{(g_0)g_1\dotsm g_{2n}g_{2n+1}H} (g_0)g_1\dotsm g_{2n}g_{2n+1} G_0
\]
for $g_i\in G_{i\pmod 2}\setminus H$.

Let $V$ be the set of vertices and $E$ the set of edges of $T$, and let $s,r\colon E\to V$ denote the source and range maps.
Given any two vertices $v,w$, there are exactly two paths between them (one starting in $v$ and ending in $w$, and one in the opposite direction),
and the length of these paths is the combinatorial distance $d(v,w)$.
A \emph{ray} in $T$ is a sequence $(x_n)_{n=0}^\infty$ of vertices,
which is geodesic in the sense that $d(x_m,x_n)=\lvert m-n\rvert$ for all $m,n$ (i.e., $x_{n+2}\neq x_n$ for all $n$).
Moreover, given two rays $(x_n)_{n=0}^\infty$ and $(y_n)_{n=0}^\infty$ in $T$, we say they are \emph{cofinal},
and write $(x_n)_{n=0}^\infty \sim (y_n)_{n=0}^\infty$ if there exist integers $k$ and $N$ such that $y_n=x_{n+k}$ for all $n>N$.
Define the boundary $\partial T$ of the Bass-Serre tree $T$ as the set of equivalence classes of cofinal rays.

For $e\in E$, define $Z_0(e)=\{v\in V : d(v,s(e))>d(v,r(e))\}$, i.e., the set of all vertices that are closer to the endpoint of $e$ than the starting point.
Moreover, define $Z_\infty(e)$ as the set of all rays $(x_n)_{n=0}^\infty$ such that $x_j=s(e)$ and $x_{j+1}=r(e)$ for some $j\geq 0$,
and then define $Z_B(e)\subset\partial T$ as $Z_\infty(e)/\sim$.
Finally, set $Z(e)=Z_0(e) \cup Z_B(e)$.
The family of all finite intersections of sets from the collection $\{Z(e) : e\in E\}$
forms a base of compact clopen sets for a totally disconnected compact Hausdorff topology on $V\cup\partial T$,
sometimes called the ``shadow topology'' on $V\cup\partial T$.
We refer to \cite[Section~4, especially Proposition~4.4]{MS} in this regard
(there it is assumed that $T$ is countable, but their proofs hold also without this hypothesis, although then the topology is not metrizable).

Moreover, by removing an edge from $T$, we get two components, so-called \emph{half-trees}.
An \emph{extended half-tree} is a half-tree together with all its associated boundary points.
In this terminology, as explained in \cite[Section~4.3]{Boudec-Bon}, the shadow topology is generated by all the extended half-trees of $V\cup\partial T$.

Next, define $F\subseteq V$ as the set of all vertices $v$ such that $s^{-1}(v)=r^{-1}(v)$ is finite, i.e., only finitely many edges start and end in $v$.
The following can be deduced from sections of \cite{Boudec-Bon,MS} mentioned above:
\begin{proposition*}[new]
The closure $\overline{\partial T}$ of $\partial T$ in $V\cup\partial T$ is $(V\setminus F)\cup\partial T$, and is compact, minimal, and $G$-invariant.
Moreover, $\partial T$ is closed in $V\cup\partial T$ if and only if $F=V$, if and only if $T$ is locally finite, if and only if $H$ has finite index in both $G_0$ and $G_1$.
\end{proposition*}

Henceforth, we fix sets $S_0$ and $S_1$ of representatives for the nontrivial left cosets of $G/G_0$ and $G/G_1$, respectively.
Then every $x\in \partial T$ can be uniquely represented by an infinite word $(g_0)g_1g_2\dotsm$,
where $g_i\in S_{i\pmod 2}$,
that is, we take either $x_0=g_0G_1$, $x_1=g_0g_1G_0$ etc.\ or $x_0=G_1$, $x_1=g_1G_0$, etc.,
i.e., every class of cofinal rays will be represented by the unique ray in the class starting with $gG_1$ for $g\in S_0\cup\{e\}$.

The boundary $\partial T$ becomes a totally disconnected locally compact Hausdorff space when equipped with the subspace topology coming from the shadow topology on $V\cup\partial T$.
This topology is generated by basic clopen sets $U((g_0)g_1\dotsm g_n)$, where $n\geq 0$ and $g_i\in S_{i\pmod 2}$,
consisting of all equivalence classes of cofinal rays that are identified with infinite words starting with $(g_0)g_1\dotsm g_n$.

The amalgam $G$ acts on its Bass-Serre tree $T$ by left translation, that is,
the action of $s\in G$ on vertices is given by $s\cdot gG_0=sgG_0$, $s\cdot gG_1=sgG_1$ and on edges by $s\cdot gH=sgH$.
Clearly, this also induces an action of $G$ on the boundary of its Bass-Serre tree $\partial T$.

\begin{lemma*}[new]
Let $g,s\in G$, and suppose that $s$ fixes $U(g)$ pointwise.
Then $sgH=gH$.
Consequently, $s$ fixes every vertex in any ray coming from an infinite word starting with $g$.
\end{lemma*}
For a complete proof that holds in a more general case, see \cite[Lemma~3.6]{BIO}.
\begin{proof}[Sketch of proof]
The idea is to assume that $sgH\neq gH$, and then construct an infinite word starting with $g$, that gives rise to a ray that is not fixed (up to cofinality) by $s$.
The argument involves division into several subcases, where the infinite word depends upon the last letter of $g$, and the first and last letter of $g^{-1}sg$.
\end{proof}

Recall \eqref{eq:K0K1} and define the set
\begin{equation}\label{eq:fixators}
K((g_0)g_1\dotsm g_n)=(g_0)g_1\dotsm g_n K_{n+1\pmod 2} g_n^{-1}\dotsm g_1^{-1}(g_0^{-1}),
\end{equation}
where $g_i\in S_{i\pmod 2}$.
This is the subgroup of $G$ consisting of all elements that fix the basic open set $U((g_0)g_1\dotsm g_n)$ pointwise.

\begin{lemma}\label{lem:interior}
We have that $\ker G = \ker(G\curvearrowright T) = \ker(G\curvearrowright \partial T) = \ker(G\curvearrowright\overline{\partial T})$,
and that $\operatorname{int}(G\curvearrowright \partial T)$ equals $\operatorname{int}(G\curvearrowright\overline{\partial T})$ and coincides with the normal closure of $K_0 \cup K_1$.
\end{lemma}

\begin{proof}
It should be clear that $\ker G=\ker(G\curvearrowright T)\subseteq\ker(G\curvearrowright \partial T)\subseteq H$.
Indeed, the latter inclusion follows from the new lemma above.
Suppose that $h\in H\setminus \ker G$.
Then there exists $g$ such that $ghg^{-1}\notin H$.
Moreover, we can find a $g$ of the form $(g_0)g_1\dotsm g_n$ with $g_i\in S_{i\pmod 2}$ with this property.
Let $x$ be any ray starting with the corresponding vertices.
Then $hx\not\sim x$, so $h\notin\ker(G\curvearrowright \partial T)\subseteq H$
The equality $\ker(G\curvearrowright\partial T)=\ker(G\curvearrowright\overline{\partial T})$ follows from continuity of the action.

To see that $\operatorname{int}(G\curvearrowright\partial T)=\operatorname{int}(G\curvearrowright\overline{\partial T})$, note first that if 
$(\overline{\partial T})^g$ has nonempty interior, i.e., there exists open nonempty $V\subseteq(\overline{\partial T})^g$,
then $V\cap\partial T\subseteq(\partial T)^g$ is open nonempty in $\partial T$.
Next, if $(\partial T)^g$ has nonempty interior, i.e., there exists open nonempty $U\subseteq(\partial T)^g$, then there exists open $V\subseteq\overline{\partial T}$ such that $U=V\cap\partial T$.
Using density and continuity of the action, it follows that $V\subseteq\overline{U}\subseteq(\overline{\partial T})^g$.

Next, $K_0$ fixes $U(g_0)$, i.e., all sequences of vertices starting with $g_0G_1$ for any $g_0\in S_0$, pointwise,
and $K_1$ fixes $U(g_1)$, i.e., all sequences of vertices starting with $G_1,g_1G_0$ for any $g_1\in S_1$, pointwise.
Hence, $K_0 \cup K_1 \subseteq \operatorname{int}(G\curvearrowright \partial T)$.
Therefore, as the latter is normal, the same inclusion holds for the normal closure of $K_0 \cup K_1$.

Pick $g\in G$ and suppose that $(\partial T)^h$ has nonempty interior.
Then $h$ must fix some basic open set pointwise, say $U(g)$ for $g=(g_0)g_1\dotsm g_n$.
This means that $h\in K(g)$, as defined in \eqref{eq:fixators}, so $h\in gK_{n+1\pmod 2}g^{-1}$,
that is, $h$ belongs to the normal closure of $K_{n+1\pmod 2}$.
Since $\operatorname{int}(G\curvearrowright \partial T)$ is generated by $\{h\in G:(\partial T)^h\text{ has nonempty interior}\}$,
the conclusion follows.

Note that $h$ also fixes $U((g_0)g_1\dotsm g_{n+1})$, so $h$ belongs to the normal closure of the other $K_i$ as well.
In fact, the normal closures of $K_0$, $K_1$, and $K_0 \cup K_1$ are all the same.
\end{proof}

\begin{definition}\label{def:int-G}
The subgroup $\operatorname{int}(G\curvearrowright\partial T)=\langle\{g\in G:(\partial T)^g\text{ has nonempty interior}\}\rangle$
of $G$, or equivalently, the normal closure of $K_0\cup K_1$ in $G$,
will be called the \emph{interior} of $G$ and denoted $\operatorname{int} G$.
\end{definition}

In Proposition~\ref{prop:cstarsimple-int} below, we show that $G$ is $C^*$-simple if and only if $\operatorname{int} G$ is $C^*$-simple,
giving an analog of Proposition~\ref{prop:amalgam ut} (however $\operatorname{int} G$ can be all of $G$).

\begin{proposition}\label{prop:slender}
The following are equivalent:
\begin{itemize}
\item[(i)] $\operatorname{int} G=\{e\}$,
\item[(ii)] $G\curvearrowright T$ is strongly faithful,
\item[(iii)] $G\curvearrowright \partial T$ is strongly faithful,
\item[(iv)] $G\curvearrowright \partial T$ is topologically free, i.e., $G\curvearrowright T$ is slender in the sense of \cite{HP}.
\end{itemize}
\end{proposition}

\begin{proof}
The equivalences between (i), (iv), and condition~(i) of Theorem~\ref{thm:trivial-Ks} follow directly from Lemma~\ref{lem:interior} and Definition~\ref{def:int-G}.
Moreover, condition~(ii) of Theorem~\ref{thm:trivial-Ks} coincides with condition~(SF) from \cite[p.~245]{Harpe}.
Remark that in \cite{Harpe}, the notations $\textup{Edg}_+X$ and $Y$ are used for $T$ and $\partial T$, respectively.
In particular, (ii) is the same as condition~(SF), and thus the above means that (ii) is equivalent with (i) and (iv).
Finally, it follows from \cite[Lemma~9]{Harpe} that (ii) implies (iii),
and it is stated immediately after the proof of \cite[Lemma~9]{Harpe} that its converse holds as well.
\end{proof}

Using the terminology of \cite{Boudec}, we see from the proof of Lemma~\ref{lem:interior}
that $K_0$ and $K_1$ are precisely the fixators of the half-trees of $T$ obtained by removing the edge $H$.

\begin{proposition}\label{prop:cstarsimple-int}
Let $G$ be a nondegenerate free product with amalgamation.
Then $G$ is $C^*$-simple if and only if $\operatorname{int} G$ is $C^*$-simple.
\end{proposition}

\begin{proof}
Suppose that $\operatorname{int} G$ is $C^*$-simple.

First, assume that $\operatorname{int} G\subseteq H$.
For the moment, write $K_i(G)$ and $K_i(G/\ker G)$ for the $K_i$'s corresponding to the amalgams $G$ and $G/\ker G$, respectively.
Clearly, a word in $G$ starts with a letter in $G_i\setminus H$ if and only if its image in $G/\ker G$ starts with a letter in $(G_i/\ker G)\setminus (H/\ker G)$.
Thus, we see that $K_i(G)/\ker G\cong K_i(G/\ker G)$.

Because $\operatorname{int} G$ is a normal subgroup of $G$ contained in $H$, we must have $\operatorname{int} G=\ker G$, and then also $K_0=K_1=\ker G$.
Thus, $G/\operatorname{int} G$ is $C^*$-simple by Theorem~\ref{thm:trivial-Ks}.
Since we have assumed that $\operatorname{int} G$ is $C^*$-simple, it follows from \cite[Theorem~1.4]{BKKO} that $G$ is $C^*$-simple.

Next, let $C_G(\operatorname{int} G)$ denote the centralizer of $\operatorname{int} G$ in $G$, and suppose that $g\in C_G(\operatorname{int} G)\setminus H$.

In particular, this means that $g$ commutes with all elements in $K_0$ and $K_1$, so $gK_ig^{-1}=K_i$ for $i=0,1$. Moreover, for all $g_i\in G_i\setminus H$,
we always have that $g_iK_{i+1\pmod 2}g_i^{-1}\supseteq K_i$.
From this it follows that $K_0=K_1$, which means that $\ker G=\operatorname{int} G$.

Indeed, let $g$ have length $n$ and denote by $g_n$ the last letter of $g$.
If $n$ is odd, then
\[
K_{n \pmod 2}=gK_{n \pmod 2}g^{-1}\subseteq\dotsb\subseteq g_nK_{n \pmod 2}g_n^{-1}\subseteq K_{n+1 \pmod 2},
\]
which easily implies $K_0 = K_1$, and hence $\ker G=\operatorname{int} G$.
If $n$ is even, then
\[
K_{n+1 \pmod 2}=gK_{n+1 \pmod 2}g^{-1}\subseteq\dotsb\subseteq g_nK_{n \pmod 2}g_n^{-1}\subseteq K_{n+1 \pmod 2},
\]
and therefore all containments are equalities.
Assume that $[G_j:H]\geq 3$, for $j=0$ or $1$, pick a letter $g_j\in G_j\setminus H$ of $g$,
and choose another element $g_j'\in G_j\setminus H$ such that $g_j^{-1}g_j'\in G_j\setminus H$.
Then
\[
(g_j')^{-1}g_jK_{j\pmod 2}g_j^{-1}g_j' = (g_j')^{-1}K_{j+1\pmod 2}g_j' \supseteq K_{j\pmod 2}.
\]
Moreover,
\[
(g_j')^{-1}g_jK_{j\pmod 2}g_j^{-1}g_j' \subseteq K_{j+1\pmod 2},\quad\text{so}\quad K_{j\pmod 2} \subseteq K_{j+1\pmod 2},
\]
meaning that $\ker G=\operatorname{int} G$.

Finally, if $C_G(\operatorname{int} G)\subseteq H$, then $C_G(\operatorname{int} G)\subseteq\ker G$, since it is normal in $G$ and contained in $H$.
Then $C_G(\operatorname{int} G)=Z(\operatorname{int} G)$, since it is contained in $\operatorname{int} G$, so it must be trivial since $\operatorname{int} G$ is assumed to be $C^*$-simple.
Hence, it follows from \cite[Theorem~1.4]{BKKO} that $G$ is $C^*$-simple.

The converse holds by \cite[Theorem~1.4]{BKKO} because $\operatorname{int} G$ is a normal subgroup of $G$.
\end{proof}

\begin{proposition}
Suppose that $\ker G$ is trivial.
Then $G$ is a weak$^*$~Powers group.
\end{proposition}

\begin{proof}
The action of $G$ on the boundary of its Bass-Serre tree is always minimal and strongly hyperbolic (see \cite[Proposition~19]{HP}).
If $\ker G$ is trivial, then the action is also faithful by Lemma~\ref{lem:interior}.
In \cite[Lemma~4]{Harpe} we now replace ``strongly faithful'' by ``faithful'',
and the first part of the proof still works, under the assumption that $F\neq\{e\}$ is a one-element set.
The rest of the argument goes along the same lines.
\end{proof}

\begin{example}
The group $\Gamma$ from Section~\ref{sec:example} is a weak$^*$~Powers group that is not $C^*$-simple.
\end{example}

We complete this section by some facts about boundary actions and refer to \cite{BKKO,KK,Ozawa} for further details.
An action of $G$ on a compact space $X$ is called a \emph{boundary action} if it is
minimal (i.e., the orbits are dense) and
strongly proximal (i.e., the orbit-closure in $P(X)$ of every probability measure on $X$ contains a point mass).
In this case we also say that the space $X$ is a $G$-boundary.

For every group $G$ there is a universal $G$-boundary called the \emph{Furstenberg boundary} and denoted $\partial_F G$.
If $X$ is any other $G$-boundary, then there exists a continuous surjective $G$-equivariant map $\partial_F G\to X$.
Moreover, we remark that for every $g\in G$, the set $(\partial_F G)^g$ is always clopen, cf.\ \cite[Lemma~3.3]{BKKO}.

In Section~\ref{sec:aish} we will often make use of the following observation from \cite{BKKO}.

\begin{lemma}\label{lem:boundary-subquotient}
Let $G$ be a group, $N$ a normal subgroup of $G$, and $L$ a normal subgroup of $N$.
Suppose that $g\in N$ is such that $(\partial_F G)^g\neq\varnothing$.
Then $(\partial_F(N/L))^{gL}\neq\varnothing$.
\end{lemma}

\begin{proof}
First, the quotient map $N\to N/L$ gives rise to an action of $N$ on $\partial_F(N/L)$,
which makes $\partial_F(N/L)$ an $N$-boundary.
Thus, there exists an $N$-equivariant continuous surjective map $\psi\colon\partial_F N\to\partial_F(N/L)$.
Next, \cite[Lemma~5.2]{BKKO} says that the $N$-action on $\partial_F N$ extends to an action of $G$,
such that $\partial_F N$ becomes a $G$-boundary.
Therefore, there is a $G$-equivariant continuous surjective map $\phi\colon\partial_F G\to\partial_F N$,
so altogether we have surjections
\[
\partial_F G\overset{\phi}{\longrightarrow}\partial_F N\overset{\psi}{\longrightarrow}\partial_F(N/L).
\]
Now, if $g\in N$ and $(\partial_F G)^g\neq\varnothing$, there exists $x\in \partial_F G$ such that $gx=x$.
It follows that $g\psi(\phi(x))=\psi(\phi(gx))=\psi(\phi(x))$, so $\psi(\phi(x))\in(\partial_F(N/L))^g=(\partial_F(N/L))^{gL}$,
which is therefore nonempty.
In particular, we have $\psi(\phi((\partial_F G)^g))\subseteq (\partial_F(N/L))^{gL}$.
\end{proof}

\begin{lemma}\label{lem:T-boundary}
Suppose that $G$ is a nondegenerate amalgam.
Then $\overline{\partial T}$ is a $G$-boundary.
\end{lemma}

\begin{proof}
The action of a nondegenerate amalgam $G=G_0*_H G_1$ on its Bass-Serre Tree $T$ is minimal and strongly hyperbolic (see \cite[Proposition~19]{HP}),
that is, of ``general type'' in the sense of \cite[Section~4.3]{Boudec-Bon}.
It follows that the action of $G$ on $\overline{\partial T}$ is minimal by \cite[Proposition~19]{HP} and Lemma~\ref{lem:interior},
extremely proximal by \cite[Proposition~4.26]{Boudec-Bon} (see \cite[Section~2.1]{Boudec-Bon} for terminology),
and thus strongly proximal by \cite[Theorem~2.3~(3.3)]{Glasner}.
Hence, $\overline{\partial T}$ is a $G$-boundary.
\end{proof}

\begin{theorem}\label{thm:K0K1-amenable}
Suppose that $K_0$ or $K_1$ is amenable.
Then $G$ is $C^*$-simple if and only if both $K_0$ and $K_1$ are trivial.
\end{theorem}

\begin{proof}
Let $x\in\partial T$ be represented by a sequence of vertices $g_0G_1$, $g_0g_1G_0$, etc., where $g_i\in S_{i\pmod 2}$
(of course, the argument is similar if it starts with $G_1$, $g_1G_0$, etc).
Then $G_x^o$ is the direct limit of the sequence $K(g_0)\subseteq K(g_0g_1)\subseteq\dotsb$, that is, of
\[
K_0 \subseteq g_0K_1g_0^{-1} \subseteq g_0g_1K_0g_1^{-1}g_0^{-1} \subseteq \dotsb.
\]
Clearly, $K_0$ is amenable if and only if $K_1$ is amenable,
since either of them is a subgroup of a conjugate of the other, see \eqref{eq:K0-K1-relation}.
Therefore, all $K((g_0)g_1\dotsm g_n)$ are also amenable, since they are conjugates of $K_0$ or $K_1$, see \eqref{eq:fixators}.
As the class of amenable groups is closed under direct limits, we have that $G_x^o$ is amenable.

Assume that $G$ is $C^*$-simple.
Since $\overline{\partial T}$ is a $G$-boundary by Lemma~\ref{lem:T-boundary} and $G$ is assumed to be $C^*$-simple,
we may use \cite[Corollary~7.5]{BKKO} to say that $C(\overline{\partial T})\rtimes_r G$ is simple.
Then it follows from \cite[Theorem~14~(2)]{Ozawa} that $G \curvearrowright \overline{\partial T}$ is topologically free,
so $G \curvearrowright \partial T$ is topologically free by Lemma~\ref{lem:interior}, that is, $\operatorname{int}G=\{e\}$.
Hence, Proposition~\ref{prop:slender} gives that both $K_0$ and $K_1$ are trivial.

Conversely, if $K_0=K_1=\{e\}$, then $G$ is $C^*$-simple by Theorem~\ref{thm:trivial-Ks}.
\end{proof}

A similar argument shows that $G_x^o$ is nonamenable for all $x\in\partial T$ if $K_0$ or $K_1$ is nonamenable,
and it seems likely that this implies $C^*$-simplicity of $G$.

\begin{remark}
Any amalgamated free product where $K_0$ and $K_1$ are nontrivial, amenable, and $K_0 \cap K_1=\{e\}$ is not $C^*$-simple, but has the unique trace property.
As noted in Remark~\ref{rmkGamma}, the group $\Gamma$ of Section~\ref{sec:example} is isomorphic to one of Le~Boudec's examples from \cite[Section~5]{Boudec}.
However, if the groups $G_0$ and $G_1$ are nonisomorphic, then such an amalgamated product will not be covered by \cite[Section~5]{Boudec}.  
\end{remark}

\section{Radical and residual classes of groups}\label{sec:radical}

In this section, by a \emph{class of groups}, we will always mean a class $X$ of groups that contains the trivial group
and is closed under isomorphisms, i.e., if $G\in X$ and $H\cong G$, then $H\in X$.
In \cite{Robinson}, this is called a \emph{group theoretical class}.

Let $X$ be a class of groups and let $G$ be any group.
Define $\rho(G)$ as the normal subgroup of $G$ generated by all normal subgroups of $G$ that belong to $X$,
and $\rho^*(G)$ as the intersection of all normal subgroups of $G$ with quotient belonging to $X$, i.e.,
\[
\begin{split}
\rho(G)&=\prod\{N : N\triangleleft G \text{ and } N\in X\},\\
\rho^*(G)&=\bigcap\{N : N\triangleleft G \text{ and } G/N\in X\}.
\end{split}
\]
These are both normal subgroups of $G$, called join and meet of the respective families.
Whenever more than one class is around, we will often write $\rho_X$ and $\rho_X^*$.

\begin{definition}\label{def:radical-residual}
A class of groups $X$ is called a \emph{radical} class if it is closed under quotients (i.e., closed under homomorphic images),
and if for any group $G$ we have
\begin{itemize}
\item[(i)] $\rho(G)\in X$,
\item[(ii)] $\rho(G/\rho(G))=\{e\}$.
\end{itemize}
A class of groups $X$ is called a \emph{residual} (or \emph{coradical} or \emph{semisimple}) class if it is closed under normal subgroups,
and if for any group $G$ we have
\begin{itemize}
\item[(i*)] $G/\rho^*(G)\in X$,
\item[(ii*)] $\rho^*(\rho^*(G))=\rho^*(G)$.
\end{itemize}
\end{definition}

Clearly, if $X$ is radical, then for any $G$ we have $\rho(\rho(G))=\rho(G)$, and $G\in X$ if and only if $\rho(G)=G$.
Moreover, if $X$ is residual, then for any $G$ we have $\rho^*(G/\rho^*(G))=\{e\}$, and $G\in X$ if and only if $\rho^*(G)=\{e\}$.

The above definitions are not completely consistent within references;
some say that a class is radical if it is closed under quotients and (i) holds,
and \emph{strict} radical if (ii) holds as well (similarly for residual),
and there are possibly other variations.

\begin{proposition}
A class of groups is radical if and only if it is closed under quotients, extensions, and satisfies \textup{(i)}.
Moreover, a class of groups is residual if and only if it is closed under normal subgroups, extensions, and satisfies \textup{(i*)}.
\end{proposition}

\begin{proof}
This is explained in \cite[Theorems~1.32 and~1.35]{Robinson},
where (i) and (i*) are equivalent to the properties $NX=X$ and $RX=X$ of \cite[p.~20 and p.~23]{Robinson}, respectively,
see \cite[p.~19]{Robinson}.
\end{proof}

Let $X$ be a class of groups closed under quotients, and define $X^*$ to be all groups satisfying $\rho(G)=\{e\}$,
that is, the class of groups with no nontrivial normal subgroups in $X$, i.e.,
\[
X^*=\{G : \rho(G)=\{e\}\}=\{G : N\triangleleft G, N\neq \{e\}\Rightarrow N\notin X\}.
\]
Similarly, if $X$ is a class of groups closed under normal subgroups, define $X_*$ to be all groups satisfying $\rho^*(G)=G$,
that is, the class of groups with no nontrivial quotients in $X$, i.e.,
\[
X_*=\{G : \rho^*(G)=G\}=\{G : N\triangleleft G, N\neq G\Rightarrow G/N\notin X\}.
\]

\begin{proposition}\label{prop:radical duality}
Let $X$ be a class of groups closed under quotients.
Then $X$ is radical if and only if $X=(X^*)_*$ if and only if $X=Y_*$ for some class of groups $Y$ that is closed under normal subgroups.

Let $X$ be a class of groups closed under normal subgroups.
Then $X$ is residual if and only if $X=(X_*)^*$ if and only if $X=Y^*$ for some class of groups $Y$ that is closed under quotients.
\end{proposition}

\begin{proof}
This follows from \cite[Theorems~1.38 and~1.39]{Robinson}, see \cite[p.~6-7]{Robinson} for notation.
\end{proof}

\begin{example}\label{ex:AR}
Let $X$ be the class of amenable groups, which is known to be a radical class.
Then \cite[Theorem~1.3]{BKKO} shows that the class of groups with the unique trace property coincides with $X^*$,
and is therefore residual by Proposition~\ref{prop:radical duality}.
Thus, a group is amenable if and only if it does not have any nontrivial quotient with the unique trace property.
\end{example}

\begin{example}
Let $X$ be the class of Powers (resp.\ weak Powers, weak$^*$~Powers) groups,
which is not closed under extensions, so it is not a residual class.
However, $X$ is closed under normal subgroups by Lemma~\ref{lem:normal Powers},
meaning that we can define the class $X_*$ of groups
with no nontrivial quotient which is a Powers (resp.\ weak Powers, weak$^*$~Powers) group.
Moreover, $X_*$ is a radical class, but $\rho_{X_*}(G)=\{e\}$ does \emph{not} imply that $G\in X$.

If $X$ is the class of all Powers groups,
then $P=(X_*)^*$ is the residual closure of $X$,
that is, the smallest residual class containing all Powers groups.
In Section~\ref{sec:aish} we will see that the class of $C^*$-simple groups is residual,
and hence it contains $P$ (it could possibly coincide with $P$).
\end{example}

\begin{example}\label{ex:NF}
Some radical classes of groups are e.g.\ locally finite groups, elementary amenable groups, amenable groups,
and groups that do not contain any nonabelian free subgroup.
We denote the latter class by $NF$,
which gives rise to the residual class $AF=(NF)^*$,
consisting of all groups for which every nontrivial normal subgroup contains a free nonabelian subgroup,
and $G\in NF$ if and only if $\rho^*_{AF}(G)=G$.
Since every amenable group belongs to $NF$, every group in $AF$ has the unique trace property.
Moreover, \eqref{eq:chain} gives that every amalgamated free product with trivial kernel belongs to $AF$.
\end{example}

Both the class of amenable groups and $NF$ are closed under subgroups,
but in general, radical classes are not necessarily closed even under normal subgroups.
A radical class that is closed under normal subgroups is sometimes called \emph{hereditary}.

\begin{lemma}\label{lem:normal radical}
Let $X$ be a class of groups satisfying (i).
Then $X$ is closed under normal subgroups if and only if for any group $G$ and normal subgroup $N$ of $G$ we have
\[
\rho(N)=\rho(G)\cap N.
\]
In particular, this implies that $\rho(N)$ is a normal subgroup of $G$.
\end{lemma}

\begin{proof}
If $G\in X$ and $N$ is normal in $G$, then $\rho(N)=\rho(G)\cap N=G\cap N=N$, so $N\in X$.

The converse is explained in \cite[Lemma~1.31, Corollaries~1 and~2]{Robinson}.
\end{proof}

A direct proof in the special case of countable amenable groups is given in \cite[Corollary~B.4]{TD}.
The result below is similar to \cite[Corollary~B.6]{TD}, using Lemma~\ref{lem:normal radical}.

For a group $G$ and a subset $N$, we let $Z_G(N)$ and $Z(G)$ denote the centralizer of $N$ in $G$,
and the center of $G$, respectively.

\begin{lemma}\label{lem:centralizer1}
Let $X$ be a class of groups that satisfies (i) and is closed under normal subgroups.
Assume that $G$ is any group and $N$ is a normal subgroup of $G$ such that $\rho(N)=\{e\}$.
Then~$\rho(G)=\rho(Z_G(N))$.
\end{lemma}

\begin{proof}
Since $\{e\}=\rho(N)=\rho(G)\cap N$, the normal subgroups $\rho(G)$ and $N$ commute, so $\rho(G)\subseteq Z_G(N)$.
Hence, $\rho(Z_G(N))=\rho(G)\cap Z_G(N)=\rho(G)$.
\end{proof}

\begin{lemma}\label{lem:centralizer2}
Let $X$ be a class of groups that satisfies (i), is closed under normal subgroups, and contains all abelian groups.
Assume that $G$ is any group and $N$ is a normal subgroup of $G$ such that both $\rho(N)$ and $\rho(G/N)$ are trivial.
Then $\rho(G)=\{e\}$.
\end{lemma}

\begin{proof}
First, $Z(N)$ is normal in $N$, so Lemma~\ref{lem:normal radical} gives that $\rho(Z(N))\subseteq\rho(N)=\{e\}$,
and since $Z(N)\in X$, we must have $Z(N)=\{e\}$.
Thus, the map $Z_G(N)\to G/N$, $x\mapsto xN$ is injective, and $\rho(G/N)=\{e\}$ implies $\rho(Z_G(N))=\{e\}$.
By Lemma~\ref{lem:centralizer1}, we thus get $\rho(G)=\{e\}$.
\end{proof}

\begin{lemma}
Let $X$ be a residual class, and suppose that $X_*$ is closed under normal subgroups and contains all abelian groups.
Let $G$ be any group and $N$ a normal subgroup of $G$.
Then $G$ belongs to $X$ if and only if both $N$ and $Z_G(N)$ belong to $X$.
\end{lemma}

\begin{proof}
Clearly, $X_*$ satisfies the conditions of Lemmas~\ref{lem:centralizer1} and~\ref{lem:centralizer2},
and by Proposition~\ref{prop:radical duality} we know that $G\in X$ if and only if $\rho_{X_*}(G)=\{e\}$.
\end{proof}

The above result is an analog of \cite[Theorem~1.4]{BKKO},
while the result below shows that when $X$ is a residual class and $G$ is a group,
then $\rho_{X_*}(G)$ is the smallest normal subgroup of $G$ that produces a quotient in $X$.

\begin{lemma}
Let $X$ be a class of groups closed under normal subgroups
and suppose that $G$ is a group with a normal subgroup $N$ such that $G/N\in X$.
Then $\rho_{X_*}(G)\subseteq N$.
\end{lemma}

\begin{proof}
If $G/N\in X$, then $\rho_{X_*}(G/N)=\{e\}$.
Suppose that $\rho_{X_*}(G)$ is not contained in $N$.
Then $\rho_{X_*}(G)/(\rho_{X_*}(G)\cap N)$ is isomorphic to $(\rho_{X_*}(G)N)/N$,
which is a normal nontrivial subgroup of $G/N$ that belongs to $X_*$.
This is a contradiction.
\end{proof}

\begin{lemma}\label{lem:finite-index}
Let $X$ be a class of groups that is closed under normal subgroups, extensions, and contains all finite groups.
Assume that $G$ is any group and $H$ is a subgroup of $G$ of finite index.
Then $G\in X$ if and only if $H\in X$.
\end{lemma}

\begin{proof}
Let $N$ be the normal core of $H$ in $G$, that is, the largest normal subgroup of $G$ contained in $H$.
It is well-known that $N$ also has finite index in $G$.
Suppose first that $G\in X$.
Then $N\in X$ and $H/N$ is finite, so $H/N\in X$, and hence $H\in X$.
The converse is similar;
if $H\in X$, then $N\in X$ and $G/N$ is finite, so $G\in X$.
\end{proof}

\section{Amenablish groups and the amenablish radical}\label{sec:aish}

Since the class of $C^*$-simple groups is closed under normal subgroups \cite[Theorem~1.4]{BKKO},
the following definition makes sense in light of Proposition~\ref{prop:radical duality}.

\begin{definition}
We call a group \emph{amenablish} if it has no nontrivial $C^*$-simple quotients.
The class of all amenablish groups is radical,
so every group $G$ has a unique maximal normal amenablish subgroup,
which will be called the \emph{amenablish radical} of $G$.
\end{definition}

It is clear that every amenable group is amenablish,
but not all amenablish groups are amenable, as explained in Corollary~\ref{cor:gamma-amenablish} below.

We will now show that the class of $C^*$-simple groups is residual,
which will imply that a group is $C^*$-simple precisely when its amenablish radical is trivial.
Since the class of $C^*$-simple groups is known to be closed under normal subgroups and extensions by \cite[Theorem~1.4]{BKKO},
we only have to prove that (i*) from Definition~\ref{def:radical-residual} holds.

\begin{proposition}\label{prop:cstarsimple-residual}
Suppose that $G$ is a group and $\{N_\alpha\}_{\alpha\in\Lambda}$ is a family of normal subgroups of $G$
such that $G/N_\alpha$ is $C^*$-simple for all $\alpha$.

Then $G/\bigcap_{\alpha\in\Lambda}N_\alpha$ is $C^*$-simple.
\end{proposition}

\begin{proof}
For any two indicies $\alpha$ and $\beta$,
the group $(N_\alpha N_\beta)/N_\beta\cong N_\alpha/(N_\alpha\cap N_\beta)$ is a normal subgroup of $G/N_\beta$,
so it is $C^*$-simple again by \cite[Theorem~1.4]{BKKO}.
Moreover, $G/(N_\alpha\cap N_\beta) \to G/N_\alpha$ is surjective with kernel $N_\alpha/(N_\alpha\cap N_\beta)$.
Hence, applying \cite[Theorem~1.4]{BKKO} once more gives that $G/(N_\alpha\cap N_\beta)$ is $C^*$-simple.
We may therefore assume that the family $\{N_\alpha\}_{\alpha\in\Lambda}$ is closed under finite intersections.

It is easy to see, using transfinite induction and the Axiom of Choice,
that we can obtain a well-ordered set $\{ N_\beta \}_{ \beta \in I }$ of normal subgroups of $G$ with the property
\[
\bigcap_{\beta \in I} N_\beta = \bigcap_{\alpha \in \Lambda} N_\alpha \overset{def}{=} N.
\]
After factoring the whole family by $N$, we deduce the following equivalent reformulation of the above statement:
Suppose $G$ is a group with a decreasing (transfinite) sequence
\[
G=N_0\supsetneq N_1\supsetneq N_2\supsetneq \dotsb \supsetneq N_\alpha \supsetneq \dotsb \supsetneq N_\beta \supsetneq \dotsb
\]
that satisfies
\begin{itemize}
\item[(i)] $N_\alpha$ is normal in $G$ for all $\alpha \in I$,
\item[(ii)] $\bigcap_{\alpha \in I} N_\alpha =\{e\}$,
\item[(iii)] $G/N_\alpha$ is $C^*$-simple for all $\alpha \in I$.
\end{itemize}
Then $G$ is $C^*$-simple.

To prove the latter statement, set $X_\alpha=\partial_F(N_\alpha/N_{\alpha+1})$,
where $\partial_F$ denotes the Furstenberg boundary.
Then $X_\alpha$ is a $G$-boundary for all $\alpha$.
Indeed, by \cite[Lemma~5.2]{BKKO} and the normality assumption,
the action of $N_\alpha/N_{\alpha+1}$ on $X_\alpha$ extends uniquely to a boundary action of $G/N_{\alpha+1}$ on $X_\alpha$.
Then, composition with the quotient map gives a boundary action of $G$ on $X_\alpha$ (compare with \cite[Proof of Theorem~1.4]{BKKO}).

Now, set $X=\prod_\alpha X_\alpha$ (with the usual product topology).
We wish to show that the action of $G$ on $X$ is a boundary action, that is, strongly proximal and minimal.
First, \cite[Lemma~3]{Ozawa} gives that the action of $G$ on $X$ is strongly proximal, since it is strongly proximal on each factor.
Next, take an arbitrary point $x=(x_\alpha)\in X$ and an arbitrary basic open set $U=\prod_\alpha U_\alpha$,
where we have $U_\alpha=X_\alpha$ except for finitely many sets $U_{\alpha_1},\dotsc,U_{\alpha_n}$,
with $\alpha_1 < \dotsb < \alpha_n$.
Note that for all $\alpha$,
the set $(N_\alpha\setminus N_{\alpha+1})x_\alpha$ is dense in $X_\alpha$ and also that $N_\alpha$ acts minimally on $X_\alpha$,
while $N_{\alpha+1}$ acts trivially on $X_\alpha$.
Therefore there exists a group element $g_1 \in N_{\alpha_1}$, such that $g_1 x_{\alpha_1} \in U_{\alpha_1}$.
Also there exists $g_2 \in N_{\alpha_2}$, such that $g_2 g_1 x_{\alpha_2} \in U_{\alpha_2}$.
We continue the argument and finally deduce that there exists $g_n \in N_{\alpha_n}$ with $g_n \cdots g_1 x_{\alpha_n} \in U_{\alpha_n}$. This shows that $g_n \cdots g_1 x \in U$.
We conclude that $Gx$ is dense in $X$,
hence, $X$ is a $G$-boundary.
Note that the last argument does not require the Axiom of Choice,
since we just need the existence of $g_i$'s and not their concrete choice.

We wish to show that this action is free, i.e., that for all nontrivial $g\in G$, the set $X^g=\{x : gx=x\}$ is empty.
So pick an arbitrary $g\in G\setminus\{e\}$.
From the assumption it follows that there must be a unique $\alpha$ such that $g\in N_\alpha\setminus N_{\alpha+1}$.
Clearly, $N_\alpha/N_{\alpha+1}$ is $C^*$-simple, since it is a normal subgroup of $G/N_{\alpha+1}$, which is $C^*$-simple by assumption,
so \cite[Theorem~6.2]{KK} and \cite[Lemma~3.3]{BKKO}
imply that $N_\alpha/N_{\alpha+1}$ acts freely on $\partial_F(N_\alpha/N_{\alpha+1})$,
that is,
\[
X_\alpha^g=\partial_F(N_\alpha/N_{\alpha+1})^{gN_{\alpha+1}}=\varnothing
\]
Hence, $X^g=\prod_\alpha X_\alpha^g=\varnothing$.

The result now follows from \cite[Theorem~6.2]{KK}.
\end{proof}

In Proposition~\ref{prop:cstarsimple-residual} the use of the Axiom of Choice can be avoided for groups that are concretely given,
e.g.\ if $G$ is given by a totally ordered set of generators and relations,
or if $G$ is given by some concrete dynamical properties.

\begin{corollary}
The class of $C^*$-simple groups is a residual class.

Hence, a group $G$ is $C^*$-simple if and only if its amenablish radical $N$ is trivial,
and $N$ is the smallest normal subgroup of $G$ that produces a $C^*$-simple quotient.
\end{corollary}

We will now describe the amenablish radical in terms of the Furstenberg boundary.

\begin{definition}\label{def:AH-radical}
Let $G$ be any group.
Set $N_0=\{e\}$ and $N_1=\operatorname{int}(G\curvearrowright\partial_F G)$, recall \eqref{eq:interior},
and moreover, for every ordinal $\alpha$, define a normal subgroup $N_{\alpha+1}$ of $G$ by
\[
N_{\alpha+1}/N_\alpha=\operatorname{int}(G/N_\alpha\curvearrowright\partial_F(G/N_\alpha)),
\]
and for every limit ordinal $\beta$, set $N_\beta=\bigcup_{\alpha<\beta}N_\alpha$, which is clearly also normal in $G$.
Then $\{N_\alpha\}_\alpha$ is an ascending normal series of $G$ which eventually stabilizes,
and we finally set $AH(G)=\bigcup_\alpha N_\alpha$.
\end{definition}

\begin{lemma}\label{lem:AH-quotient}
For any group $G$, the quotient $G/AH(G)$ is $C^*$-simple.
\end{lemma}

\begin{proof}
Let $\{N_\alpha\}_\alpha$ be as in Definition~\ref{def:AH-radical}.
Then there exists some ordinal $\beta$ such that $AH(G)=N_\beta$.
If $G/N_\beta$ is not $C^*$-simple,
then $N_{\beta+1}/N_\beta=(\partial_F(G/N_\beta))^{gN_\beta}\neq\varnothing$ is nontrivial.
Hence, $N_\beta\subsetneq N_{\beta+1}$, contradicting the definition of $AH(G)$.
\end{proof}

Note that $G/\operatorname{int}(G\curvearrowright\partial_F G)$ is not necessarily $C^*$-simple, i.e.,
$AH(G)$ is in general bigger than $\operatorname{int}(G\curvearrowright\partial_F G)$.
Indeed, it was explained to us by Adrien Le~Boudec that by applying \cite[Theorem~1.11]{Boudec-Bon},
one can construct an example $G=G(F,F')$ such that $\operatorname{int}(G\curvearrowright\partial_F G)$ has index two in $G$
(the condition is that $F'$ is generated by its point stabilizers).
We refer to \cite{Boudec2,Boudec,Boudec-Bon} for more about this construction.

\begin{lemma}\label{lem:simple-quotient}
Suppose that $N$ is a normal subgroup of a group $G$ such that $G/N$ is $C^*$-simple.
Then $AH(G)\subseteq N$.
\end{lemma}

\begin{proof}
The action of $G/N$ on $X=\partial_F(G/N)$ is free by \cite[Theorem~6.2]{KK}.
Pick $g\in\operatorname{int}(G\curvearrowright\partial_F G)$ such that $(\partial_F G)^g\neq\varnothing$.
By Lemma~\ref{lem:boundary-subquotient}, we have $X^{gN}\neq\varnothing$, which means that $gN$ is trivial in $G/N$, i.e., $g\in N$.
Since the set of all $g$ with $(\partial_F G)^g\neq\varnothing$ generates $\operatorname{int}(G\curvearrowright\partial_F G)$,
it follows that $\operatorname{int}(G\curvearrowright\partial_F G)\subseteq N$.

We continue by transfinite induction.
Let $\{N_\alpha\}_\alpha$ be the series from Definition~\ref{def:AH-radical} associated with $G$.
We have shown that $N_1\subseteq N$.
Assume next that $N_\alpha\subseteq N$ for some ordinal $\alpha$ and note that there is a quotient map $G/N_\alpha\to G/N$.
Choose $g\in N_{\alpha+1}$ such that $(\partial_F(G/N_\alpha))^{gN_\alpha}\neq\varnothing$.
Then the same argument as above gives that $X^{gN}\neq\varnothing$, so $g\in N$.
Hence, we conclude that $N_{\alpha+1}\subseteq N$.
Finally, if $\beta$ is a limit ordinal and $N_\alpha\subseteq N$ for all $\alpha<\beta$, then clearly $N_\beta\subseteq N$.
\end{proof}

\begin{lemma}\label{lem:normal-aish}
Let $N$ be a normal subgroup of $G$.
Then $AH(N)=AH(G)\cap N$.
\end{lemma}

\begin{proof}
Pick $g\in \operatorname{int}(G\curvearrowright\partial_F G)\cap N$ such that $(\partial_F G)^g\neq\varnothing$.
It follows from Lemma~\ref{lem:boundary-subquotient} that $(\partial_F N)^g\neq\varnothing$,
so $g\in\operatorname{int}(N\curvearrowright\partial_F N)$.
Since the set of all $g$ with $(\partial_F G)^g\neq\varnothing$ generates $\operatorname{int}(G\curvearrowright\partial_F G)$,
we get that $\operatorname{int}(G\curvearrowright\partial_F G)\cap N\subseteq\operatorname{int}(N\curvearrowright\partial_F N)$.

We continue by transfinite induction.
Let $\{N_\alpha\}_\alpha$ and $\{H_\alpha\}_\alpha$ be the series from Definition~\ref{def:AH-radical} associated with $G$ and $N$,
respectively.
We have shown that $N_1\cap N\subseteq H_1$.
Let $\alpha$ be an ordinal number and assume that $N_\alpha\cap N\subseteq H_\alpha$.
Note that $N/(N\cap N_\alpha)\cong (NN_\alpha)/N_\alpha$ is a normal subgroup of $G/N_\alpha$,
and that $N/H_\alpha$ is a quotient of $N/(N\cap N_\alpha)$.
Choose $g\in N_{\alpha+1}\cap N$ such that $\partial_F(G/N_\alpha)^{gN_\alpha}\neq\varnothing$.
Then by Lemma~\ref{lem:boundary-subquotient} we have $(\partial_F(N/H_\alpha))^{gH_\alpha}\neq\varnothing$.
Hence, $gH_\alpha\in\operatorname{int}(N/H_\alpha\curvearrowright\partial_F(N/H_\alpha))=H_{\alpha+1}/H_\alpha$,
and it follows that $N_{\alpha+1}\cap N\subseteq H_{\alpha+1}$.
Finally, if $\beta$ is a limit ordinal and $N_\alpha\cap N\subseteq H_\alpha$ for all $\alpha<\beta$,
then clearly $N_\beta\cap N\subseteq H_\beta$.
Thus, we have $AH(G)\cap N\subseteq AH(N)$.

For the opposite inclusion, note that $G/AH(G)$ is $C^*$-simple by Lemma~\ref{lem:AH-quotient}, and that $(AH(G)N)/AH(G)$ is normal in $G/AH(G)$,
so $N/(AH(G)\cap N)\cong (AH(G)N)/AH(G)$ is $C^*$-simple by using \cite[Theorem~1.4]{BKKO}.
Hence, Lemma~\ref{lem:simple-quotient} gives that $AH(N)\subseteq AH(G)\cap N$.
\end{proof}

\begin{proposition}\label{prop:aish-rad}
For any group $G$, the amenablish radical of $G$ coincides with $AH(G)$.
\end{proposition}

\begin{proof}
We need to show that $AH(G)$ is amenablish,
and that it contains all normal amenablish subgroups of $G$.

Set $M=\operatorname{int}(G\curvearrowright\partial_F G)$.
Suppose first that $L$ is a normal subgroup of $M$ such that $L\neq M$.
Pick $g\in M\setminus L$ so that $(\partial_F G)^g\neq\varnothing$.
It follows from Lemma~\ref{lem:boundary-subquotient} that $(\partial_F(M/L))^{gL}\neq\varnothing$, so $M/L$ is not $C^*$-simple.
Hence, $M$ is amenablish.

Let $\{N_\alpha\}_\alpha$ be the series from Definition~\ref{def:AH-radical} associated with $G$.
Then it follows that $N_{\alpha+1}/N_\alpha$ is amenablish for every ordinal $\alpha$,
by using the same argument as above with $G/N_\alpha$ in place of $G$.
Since the class of amenablish groups is radical,
it is closed under extensions and under increasing unions of normal subgroups.
Since $N_1$ is amenablish, an argument by transfinite induction gives that $N_\alpha$ is amenablish for every ordinal $\alpha$.
Hence, $AH(G)$ is amenablish.

Next, let $L$ be an amenablish normal subgroup of $G$, and assume that $L$ is not contained in $AH(G)$.
Set $K=L\cap AH(G)$, then $K\neq L$ and $K=AH(L)$ by Lemma~\ref{lem:normal-aish}.
Hence, $L/K$ is $C^*$-simple by Lemma~\ref{lem:AH-quotient}.
\end{proof}

\begin{lemma}\label{lem:normal aish}
The class of amenablish groups is closed under normal subgroups.
\end{lemma}

\begin{proof}
This is a direct consequence of Lemma~\ref{lem:normal radical}, Lemma~\ref{lem:normal-aish}, and Proposition~\ref{prop:aish-rad}.
\end{proof}

\begin{lemma}\label{lem:finite index aish}
Let $G$ be any group and $H$ a subgroup of finite index.
Then $G$ is amenablish if and only if $H$ is amenablish.
\end{lemma}

\begin{proof}
The class of amenablish groups is closed under normal subgroups, extensions, and contains all finite groups.
Hence, Lemma~\ref{lem:finite-index} applies.
\end{proof}

\begin{corollary}\label{cor:gamma-amenablish}
The group $\Gamma$ of Section~\ref{sec:example} is amenablish, but not amenable (and has trivial amenable radical).
\end{corollary}

\begin{proof}
By Theorem~\ref{thm:gamma-nonsimple}, the group $\Gamma$ is not $C^*$-simple, but has the unique trace property,
so it has trivial amenable radical and is icc.
The normal subgroup $\Gamma'$ from Proposition~\ref{prop:simple-by-finite} is not $C^*$-simple either,
because it has finite index in $\Gamma$ (see \cite[Proposition~19~(iv)]{Harpe2}).
Since $\Gamma'$ is simple and $AH(\Gamma')\neq\{e\}$, we must have $AH(\Gamma')=\Gamma'$, that is, $\Gamma'$ is amenablish.
Hence, it follows from Lemma~\ref{lem:finite index aish} that $\Gamma$ is amenablish.
\end{proof}

\begin{remark}
The class of amenablish groups is not closed under subgroups.
Indeed, by Corollary~\ref{cor:gamma-amenablish} the group $\Gamma$ is amenablish,
but it contains as subgroup a nonabelian free group, which is $C^*$-simple, and thus not amenablish.

Moreover, $AH(\Gamma)=\Gamma$ and $NF(\Gamma)=\{e\}$,
while recent work by Olshanskii and Osin \cite{O-O} presents a group $G$ with the property $AH(G)=\{e\}$ and $NF(G)=G$.
Hence, it seems to be no relation between the class of amenablish groups and $NF$ (except that both contain all amenable groups).
\end{remark}

\begin{remark}\label{AG}
Let $G$ be a countable group, and let $N$ be the subgroup of $G$ generated by all recurrent amenable subgroups in $G$ (see \cite[Definition~5.1]{Kennedy}). Does $N$ coincide with $AH(G)$?

It is mentioned in \cite[Remark~5.4]{Kennedy} that for every $x\in\partial_F G$ the subgroup $G_x$ is recurrent amenable in $G$,
so $\operatorname{int}(G\curvearrowright\partial_F G)=\langle G_x : x\in\partial_F G\rangle\subseteq N$.
Moreover, \cite[Theorem~1.1]{Kennedy} says that $N$ is trivial if and only if $G$ is $C^*$-simple,
that is, if and only if $\operatorname{int}(G\curvearrowright\partial_F G)$ is trivial.

Note also that a recent paper \cite[Section~4]{Boudec-Bon} introduces a unique maximal amenable uniformly recurrent subgroup $\mathcal{A}_G$ of $G$,
and \cite[Proposition~2.21~(ii)]{Boudec-Bon} states that $\langle H:H\in\mathcal{A}_G\rangle=\operatorname{int}(G\curvearrowright\partial_F G)$,
which in general is smaller than $AH(G)$, cf.\ comment after Lemma~\ref{lem:AH-quotient}.

\end{remark}

\begin{example}
If $G$ is a simple group, then $G$ is either $C^*$-simple or amenablish.
E.g.\ Thompson's group $T$ is known to be simple,
so it follows directly from \cite{Boudec-Bon,Haagerup-Olesen} that $T$ is amenablish if and only if Thompson's group $F$ is amenable.
\end{example}

\section*{Acknowledgements}

The second author is funded by the Research Council of Norway through FRINATEK, project no.~240913.
Part of the work of the second author was conducted during a research stay at Arizona State University, Fall 2015,
and another part when attending the research program ``Classification of operator algebras: complexity, rigidity, and dynamics''
at the Mittag-Leffler Institute outside Stockholm, January--February 2016.
In both cases, he would like to thank the hosts for providing excellent working conditions.
The second author is also grateful to Erik B\'edos for useful feedback on an early version of the article,
and to Adrien Le~Boudec for helpful e-mail correspondence at a later stage.
Finally, the authors would like to thank the referee for several valuable comments.

\section*{Update November 2017}

This paper was published in Journal of Functional Analysis in 2017 (272(9):3712--3741).
After that, it was pointed out to us by Rasmus Sylvester Bryder that we had incorrectly assumed that the boundary of the Bass-Serre tree is always compact,
and therefore the statement of Lemma~\ref{lem:T-boundary} was incorrect.
This did not cause any major problems, and the proof of the only result that depended on it, Theorem~\ref{thm:K0K1-amenable},
was easily fixable (it even turned out that the result can be generalized, see \cite[Theorem~3.9]{BIO}).
We still decided to reformulate certain paragraphs of Section~\ref{sec:actions} to accomodate for this mistake,
by inserting a new proposition and modify Lemma~\ref{lem:T-boundary} and the proof of Theorem~\ref{thm:K0K1-amenable}.
At the same time, we also inserted a new lemma used to clarify the proof of Lemma~\ref{lem:interior},
fixed an inaccuracy in the proof of Theorem~\ref{thm:trivial-Ks}, corrected some typos, and updated the reference list.

\bibliographystyle{plain}

\end{document}